\documentclass{amsart}

\newtheorem{thm}{Theorem}[section]

\newtheorem{lem}[thm]{Lemma}
\newtheorem{prop}[thm]{Proposition}
\newtheorem{criterion}[thm]{Criterion}
\newtheorem{subcriterion}[thm]{Sub-criterion}
\theoremstyle{definition}
\newtheorem{defn}[thm]{Definition}
\newtheorem{example}[thm]{Example}

\theoremstyle{remark}
\newtheorem{rem}[thm]{Remark}
\newtheorem{Question}[thm]{Question}
\numberwithin{equation}{section}
\usepackage{amsmath,amssymb,graphicx,bbm,amsthm,hyperref,mathrsfs,tikz,palatino,bm,enumitem,color,ifpdf}
\usetikzlibrary{shapes.geometric,arrows}
\begin{document}

\title[]
{New criteria and Constructions of Brunnian Links}

\author{Sheng Bai}
\address{School of Mathematical Sciences, Peking University}
\email{barries@163.com}

\author{Weibiao Wang}
\address{School of Mathematical Sciences, Peking University}
\email{wwb@pku.edu.cn}



\begin{abstract}
  We present two practical and widely applicable methods, including some criteria and a general procedure, for detecting Brunnian property of a link, if each component is known to be unknot. The methods are based on observation and handwork. They are used successfully for all Brunnian links known so far. Typical examples and extensive experiments illustrate their efficiency. As an application, infinite families of Brunnian links are created, and we establish a general way to construct new ones in bulk.
\end{abstract}

\date{2020 June. 16}
\maketitle

\tableofcontents

\section{Introduction}\label{sect:introduction}

Detecting unlink is generally considered as difficult at least as detecting unknot, which is a famous problem in knot theory. The most common methods are link invariants. These rich and powerful tools, though, often become hard to compute by hand when links are large. A question that has always attracted the authors is:

\begin{Question}
	Can we establish an intuitive and rigorous method to detect unlinks, under some assumptions, such as some components are known to be unknotted?
\end{Question}

It is still unknown whether HOMFLY polynomial detects unlink (also whether Jones polynomial detects unknot, c.f. \cite{EKT}). If one expects a negative answer and tries to seek for a counterexample, one will need to prove the desired link is nontrivial. At this point, an intuitive method might be more enlightening than looking for other invariants.

We are concerned in this paper with Brunnian links.

\begin{defn}\label{def:brunn}
	A link $L$ of $n$ $(>1)$ components in $\mathbb{R}^3$  is \emph{Brunnian} if it satisfies the following properties:
	\begin{enumerate}
		\item[(ST):] every sublink with $(n-1)$ components is trivial;
		\item[(NT):] $L$ itself is nontrivial.
	\end{enumerate}
\end{defn}

Analyzing a link, we may naturally consider its nontrivial sublinks, the minimal ones among whom are exactly knots and Brunnian links. In this sense, they are of natural importance in study of links. For theoretic results of Brunnian links the reader is referred to \cite{MS,MY,SH}. The first examples were introduced by H. Brunn\cite{B} in 1892 but it was not until 1961 their Brunnian property was proved by DeBrunner\cite{HDb}. Since then, various authors constructed many kinds of examples, which, as far as we know, all appear in \cite{D.Rolfsen, BFJRVZ, BS, BCS, L, J, F, M, MM, TKa}. Recently, Nils A. Baas et al.\cite{B10, B12, BC, BCS, BS} constructed several infinite families of Brunnian links, vastly expanding our knowledge in this area. However, in these papers, their Brunnian property is verified by general invariants, not taking advantage of the Brunnian property and often difficult in handwork.

The purpose of this paper is to give convenient and general approach to detecting Brunnian property of links. We focus on the following question:

\begin{Question}
	Suppose a link $L$ possesses the property (ST), how can we check the property (NT)?
\end{Question}

In fact, if all component of a link $L$ are known to be unknots, we can detect its Brunnian property, by applying the method for this question to all sublinks with $k$ components and induction on $k$.

We propose two methods, named Arc-method and Circle-method, including several criteria and a general procedure. The novel methods, significantly different from using invariants, depend crucially on observation and are worked by hand. They are practical and turn out to work well for all known Brunnian links. The efficiency will be demonstrated by typical examples. Arc-method is very fast when it works. The power of Circle-method will be reflected in large series of links and the flexibility of Circle-method leads us to construct new Brunnian links. We will exhibit infinitely many infinite families of new Brunnian links, far more than the existing ones.

\begin{figure}[htbp]
	\centering
	\ifpdf
	\setlength{\unitlength}{1bp}%
	\begin{picture}(186.30, 81.46)(0,0)
	\put(0,0){\includegraphics{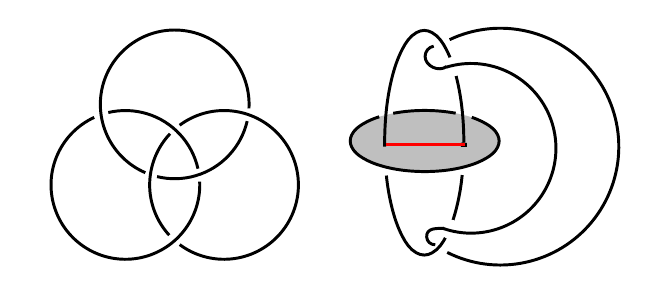}}
	\put(5.67,11.98){\fontsize{8.03}{9.64}\selectfont $C_2$}
	\put(84.32,11.98){\fontsize{8.03}{9.64}\selectfont $C_3$}
	\put(23.27,69.52){\fontsize{8.03}{9.64}\selectfont $C_1$}
	\put(97.09,49.05){\fontsize{8.03}{9.64}\selectfont $C_1$}
	\put(104.09,11.98){\fontsize{8.03}{9.64}\selectfont $C_2$}
	\put(171.29,11.98){\fontsize{8.03}{9.64}\selectfont $C_3$}
	\end{picture}%
	\else
	\setlength{\unitlength}{1bp}%
	\begin{picture}(186.30, 81.46)(0,0)
	\put(0,0){\includegraphics{borromean}}
	\put(5.67,11.98){\fontsize{8.03}{9.64}\selectfont $C_2$}
	\put(84.32,11.98){\fontsize{8.03}{9.64}\selectfont $C_3$}
	\put(23.27,69.52){\fontsize{8.03}{9.64}\selectfont $C_1$}
	\put(97.09,49.05){\fontsize{8.03}{9.64}\selectfont $C_1$}
	\put(104.09,11.98){\fontsize{8.03}{9.64}\selectfont $C_2$}
	\put(171.29,11.98){\fontsize{8.03}{9.64}\selectfont $C_3$}
	\end{picture}%
	\fi
	\caption{\label{fig:borromean}%
		Borromean ring}
\end{figure}

Let us illustrate our natural idea with an example. The two links in Fig. \ref{fig:borromean} are both Borromean rings. We use the right one. Take the gray disk $D$ bounded by component $C_1$, which intersects $C_2$ at two points. Connect these two points by an arc on $D$. Each component of $C_2- D$ together with the red arc forms a circle which can not bound a disk avoiding $C_3$ due to linking number. Arc-criterion then asserts Borromean ring is nontrivial.

\section{Preliminaries}\label{sect:preliminaries}

We work in smooth category. Without special explanation, we always consider links in $\mathbb{R}^3$. The notations $C,L,D$ (maybe with subscripts) are used always to denote an unknot, a link and an embedded disk respectively. The interior of a manifold is denoted by ${\rm Int}(\cdot)$. Each arc is assumed to be simple, and intersections are always assumed to be compact and transverse.

The following proposition is of fundamental importance for us, which follows immediately from Lemma \ref{lem:basic}.

\begin{prop}\label{prop:basic}
	Let $L$ be a link satisfying property (ST) and $L_1$ be any proper sublink of $L$. Then $L$ satisfies (NT) if and only if $L_1$ can not bound mutually disjoint disks in the complement of $L-L_1$.
\end{prop}

\begin{lem}\label{lem:basic}
	Let $L=\cup _{i=1}^n C_i$ be an $n$-component link satisfying (ST). Then $L$ is trivial if and only if $C_1$ bounds a disk in the complement of $L-C_1$.
\end{lem}

\begin{proof}
	The ``only if" part is obvious. For the converse, let $D_1$ be a closed disk with $\partial D_1=C_1$ such that ${\rm Int}(D_1)\cap L=\emptyset$. Retract $D_1$ to be small enough, and move it far away by isotopy. By (ST), $C_i, i=2,...,n$ bounds disjoint disks, and those disks can be assumed to be far from the now $D_1$.  We can then move $D_1$ to the original position, while using isotopy extension to $\mathbb{R}^3 \backslash \cup_{i=2}^n C_i$. So $C_i, i=1,...,n$ bounds mutually disjoint disks and by definition $L$ is trivial.
\end{proof}

We now describe our basic tool and notions. Suppose $L=\cup_{i=1}^n C_i$ ($n>1$) is an $n$-component link with property (ST). As a tool, choose a proper subset $I \subset \{1,2,\cdots,n\}$ and take disjoint disks $D_i (i\in I)$ bounded by $C_i (i\in I)$ with $C_i =\partial D_i$. Then $L$ is divided into three disjoint sublinks, $L_I \triangleq \cup_{i \in I} C_i$, $L_J \triangleq \cup_{j \in J} C_j$, and $L_0  \triangleq  L- L_I -L_J$, where $L_J$ consists of the components intersecting $\cup_{i \in I} {\rm Int}(D_i)$, and $L_0$ may be empty. If $L_J =\emptyset$, by Proposition \ref{prop:basic}, $L$ is trivial. So we always assume $L_J \neq\emptyset$.

{\sc test disk and thicken disk}: For each $i\in I$, we call $D_i$ a \emph{test disk}. Now $\cup_{i \in I} D_i$ intersects $L_J$ with finitely many points $\{ p_t \}_{t\in T}$. Take an embedded cylinder $\mathbf{D_i} \cong D \times [0,1]$ such that $D \times \{\frac{1}{2}\} =D_i$, and $\mathbf{D_i}$ intersects $L_J$ with several disjoint arcs each of which takes one $p_t$ as interior point (see Fig. \ref{fig:testcomplex}). Choose $\mathbf{D_i} (i\in I)$ to be ``thin'' enough so that they are mutually disjoint. We call $\mathbf{D_i}$ a \emph{thicken disk} of $D_i$. It has two sides:
\begin{align}
   D_i^+ \triangleq D \times \{ 1 \}; D_i^- \triangleq D \times \{ 0 \}. \nonumber
\end{align}

{\sc test complex and test subcomplex}: We call $U  \triangleq  \cup_{i \in I} \mathbf{D_i} \cup L$ a \emph{test complex}. A subset $U'\subset U$ is called a \emph{test subcomplex of $U$} if it consists of some thicken disks and a closed subset of $L$.

\section{Arc-method}\label{sect:arc}

\subsection{Arc-Criterion}\label{subsect:criterionA}\

Let $L$ be a link satisfying (ST) and divided into sublinks $L_I$, $L_J$, and $L_0$ with respect to test disks $\{D_i\}_{i\in I}$ as in Preliminaries. Notice that for any $j \in J$, $C_j-\cup_{i \in I}{\rm Int}( \mathbf{D_i})$ is a disjoint union of arcs.

\begin{defn}\label{def:ear}
	If an arc component $e$ of $C_j-\cup_{i \in I}{\rm Int}( \mathbf{D_i})$ has two endpoints both on $D_i^\sigma$ ($\sigma$ can be + or $-$), then $e$ is an \emph{ear} of $C_j$ on $D_i^\sigma$.
\end{defn}

For an ear $e$ of $C_j$ on $D_i^\sigma$ and a test subcomplex $U'$ containing $e$, if there exists an arc $\alpha \subset D_i^\sigma$ (avoiding $L$) connecting the endpoints of $e$, such that $e \cup \alpha$ bounds a disk $D_e$ with ${\rm Int}(D_e)\subset \mathbb{R}^3-U'$, then we say $e$ is \emph{compressible for $U'$} and call $D_e$ a \emph{compressing disk of $e$ for $U'$}.
When $U'=L$, the arc $\alpha$ is called an \emph{incredible arc of $e$}.

\begin{criterion}\label{criterionA}(\textbf{Arc-criterion})
	Suppose an $n$-component link $L=\cup_{i=1}^n C_i$ satisfies (ST) and is divided into sublinks $L_I$, $L_J$, and $L_0$ with respect to test disks $\{D_i\}_{i\in I}$ as in Preliminaries. If for some $j \in J$, there is at most one ear of $C_j$ compressible for $L$, then $L$ is Brunnian.
\end{criterion}

The criterion is a direct conclusion of the following proposition. We call it Arc-criterion for it involves incredible arcs to detect Brunnian property, and Circle-criterion is named for incredible circles (see Section \ref{sect:circle}).

\begin{prop}\label{prop:A}
	Suppose an $n$-component link $L=\cup_{i=1}^n C_i$ satisfies (ST) and is divided into sublinks $L_I$, $L_J$, and $L_0$ with respect to test disks $\{D_i\}_{i\in I}$ as in Preliminaries. If $L$ is trivial, then for each $j \in J$, there are at least 2 ears of $C_j$ compressible for $L$.
\end{prop}

\begin{proof}
	By Proposition \ref{prop:basic}, if $L$ is trivial, then $L_J$ bounds a disjoint union of disks $\sqcup_{j \in J} D_j$ with ${\rm Int}(D_j)\cap L=\emptyset$, for $\forall j\in J$. Fix $j\in J$ and we see that $D_j\cap \cup_{i\in I} D_i$ is a disjoint union of circles and arcs. Since $D_j\cap L_I=\emptyset$, the endpoints of the arcs are all on $C_j$, and all these arcs cut $D_j$ into some regions, each isomorphic to a disk. There are at least two such regions, say $R_1,R_2$, outermost on $D_j$ (i.e., the boundary of $R_s(s=1,2)$ is a union of an arc on $C_j$ and an arc component of $D_j\cap\cup_{i\in I}D_i$, and $\mathrm{Int}(R_s)\cap \cup_{i\in I} C_i=\emptyset$), as shown in Fig.  \ref{fig:region}. The boundaries of these regions provide the ears and their incredible arcs.
\end{proof}

\begin{figure}[htbp]
	\centering
	\ifpdf
	\setlength{\unitlength}{1bp}%
	\begin{picture}(85.33, 85.33)(0,0)
	\put(0,0){\includegraphics{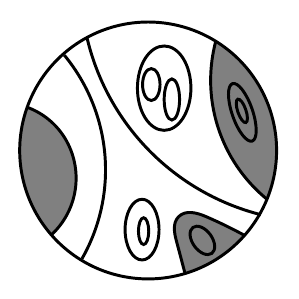}}
	\end{picture}%
	\else
	\setlength{\unitlength}{1bp}%
	\begin{picture}(85.33, 85.33)(0,0)
	\put(0,0){\includegraphics{region}}
	\end{picture}%
	\fi
	\caption{\label{fig:region}%
		Outermost regions on $D_j$}
\end{figure}

The proposition provides a necessary and sufficient condition for $L$ to be trivial. In fact, if an ear $e$ on $D_i^+$ or $D_i^-$ is compressible for $L$ with compressing disk $D_e$, we can push $e$ along $D_e$, across $D_i$ to eliminate two intersection points of $C_j\cap D_i$. It induces an isotopy of $L$, and $L_I,L_J$ may get adjusted with respect to the test disks $D_i,i\in I$. Repeat it inductively. If finally $L_J$ becomes empty, by Proposition \ref{prop:basic} $L$ is trivial, otherwise Criterion \ref{criterionA} tells it is Brunnian. Therefore, given a link consisting of unknots, theoretically we can check every sublink with Arc-criterion inductively and finally decide whether it is Brunnian.

\subsection{Example analysis}\label{subsect:examplesA}\

Suppose a link satisfies (ST). To prove it is Brunnian by Arc-method, we shall take appropriate test disks and then verify the incompressibility of enough ears by investigating all candidates for their incredible arcs. We present some examples here to show the way it works and the efficiency it possesses.

\begin{example}\label{example:communi}
\begin{figure}[htbp]
    \centering
    \includegraphics[height=3.6cm]{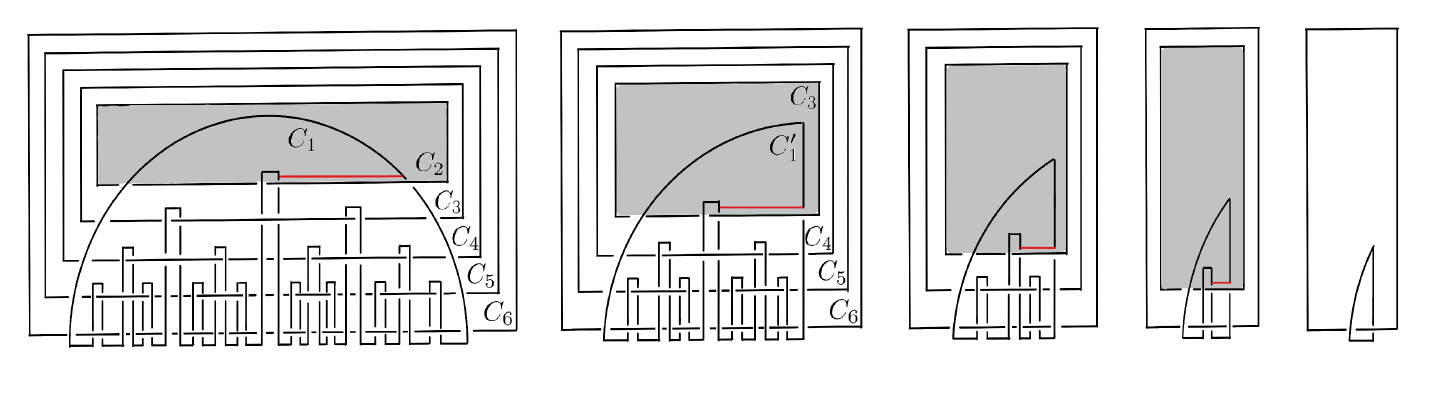}
    \caption{Example \ref{example:communi}}
    \label{fig:communi}
\end{figure}

The first link in Fig. \ref{fig:communi} has 6 components, and it is easy to see that each sublink with 5 components is trivial. Now we apply Arc-criterion to show it is a Brunnian link. Take the test disk colored in gray and $L_I=C_2,L_J=C_1,L_0=\bigcup_{k=3}^6 C_k$. There are only two ears, denoted $e_1,e_2$. The candidate for incredible arc of $e_1$ is unique up to isotopy, colored in red and denoted $\alpha$. By Criterion \ref{criterionA}, we only need to show the circle $C_1'\triangleq e_1\cup\alpha$ can not bound a disk in the complement of $L_0$, and by Proposition \ref{prop:basic}, it is equivalent to that the second link in Fig.  \ref{fig:communi} is Brunnian. Inductively, it suffices to show the rightmost link is Brunnian, which is obvious.
\end{example}

\begin{figure}[htbp]
	\centering
	\ifpdf
	\setlength{\unitlength}{1bp}%
	\begin{picture}(270.67, 91.98)(0,0)
	\put(0,0){\includegraphics{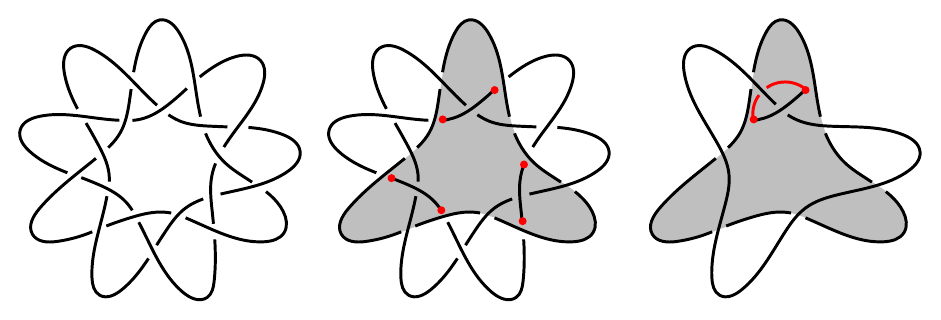}}
	\put(8.12,13.75){\fontsize{8.83}{10.60}\selectfont $C_1$}
	\put(7.96,49.62){\fontsize{8.83}{10.60}\selectfont $C_2$}
	\put(7.44,76.01){\fontsize{8.83}{10.60}\selectfont $C_3$}
	\put(185.47,76.01){\fontsize{8.83}{10.60}\selectfont $C_3$}
	\put(219.17,51.85){\fontsize{8.83}{10.60}\selectfont $e$}
	\put(222.42,69.55){\fontsize{8.83}{10.60}\selectfont \textcolor[rgb]{1, 0, 0}{$\alpha$}}
	\end{picture}%
	\else
	\setlength{\unitlength}{1bp}%
	\begin{picture}(270.67, 91.98)(0,0)
	\put(0,0){\includegraphics{gBr}}
	\put(8.12,13.75){\fontsize{8.83}{10.60}\selectfont $C_1$}
	\put(7.96,49.62){\fontsize{8.83}{10.60}\selectfont $C_2$}
	\put(7.44,76.01){\fontsize{8.83}{10.60}\selectfont $C_3$}
	\put(185.47,76.01){\fontsize{8.83}{10.60}\selectfont $C_3$}
	\put(219.17,51.85){\fontsize{8.83}{10.60}\selectfont $e$}
	\put(222.42,69.55){\fontsize{8.83}{10.60}\selectfont \textcolor[rgb]{1, 0, 0}{$\alpha$}}
	\end{picture}%
	\fi
	\caption{\label{fig:gBr}%
		Generalized Borrommean ring}
\end{figure}
\begin{example}
	(Fig.  \ref{fig:gBr}) Take the gray test disk. By symmetry, we just consider one ear, denoted $e$. Delete all other ears, then up to isotopy, there is only one way, as the arc $\alpha$, to connect the two endpoints of $e$ on the disk. The circle $e\cup\alpha$ has linking number 1 with $C_3$, thus $e$ is incompressible for the link.
\end{example}

For an ear $e$ on $D_i^\sigma$ ($\sigma$ is + or $-$), connect the endpoints by an arc $\alpha$ in $D_i^\sigma$. If $e \cup \alpha$ is a nontrivial knot, we say $e$ is \emph{knotted}, otherwise \emph{unknotted}. It is independent of the choice of $\alpha$. Obviously, a compressible ear must be unknotted.

\begin{example}\label{example:Brunnian braid}
(Fig.  \ref{fig:F})
	Take the gray test disk. There are four ears, denoted $a,b,c,d$. With the same argument as in Example \ref{example:gBr}, $b,d$ are incompressible. And $a$ is knotted as shown on the right, hence incompressible.

\begin{figure}[htbp]
	\centering
	\ifpdf
	\setlength{\unitlength}{1bp}%
	\begin{picture}(316.73, 75.37)(0,0)
	\put(0,0){\includegraphics{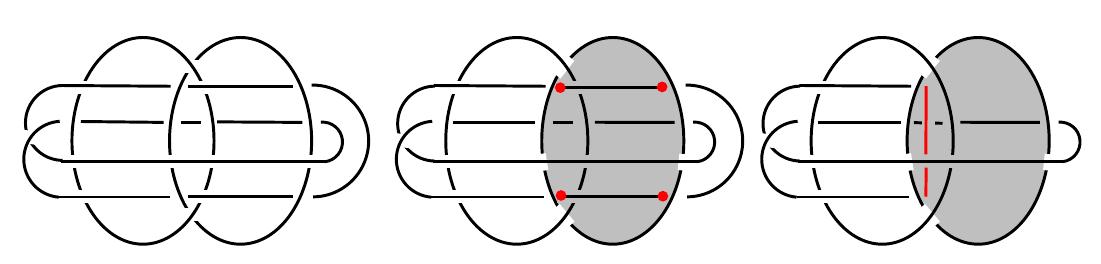}}
	\put(114.83,52.34){\fontsize{9.72}{11.66}\selectfont $a$}
	\put(174.97,52.34){\fontsize{9.72}{11.66}\selectfont $b$}
	\put(174.97,11.00){\fontsize{9.72}{11.66}\selectfont $d$}
	\put(205.03,52.34){\fontsize{9.72}{11.66}\selectfont $c$}
	\put(18.45,62.12){\fontsize{9.72}{11.66}\selectfont $C_1$}
	\put(80.08,62.12){\fontsize{9.72}{11.66}\selectfont $C_2$}
	\put(6.80,9.50){\fontsize{9.72}{11.66}\selectfont $C_3$}
	\put(220.06,52.34){\fontsize{9.72}{11.66}\selectfont $a$}
	\end{picture}%
	\else
	\setlength{\unitlength}{1bp}%
	\begin{picture}(316.73, 75.37)(0,0)
	\put(0,0){\includegraphics{F}}
	\put(114.83,52.34){\fontsize{9.72}{11.66}\selectfont $a$}
	\put(174.97,52.34){\fontsize{9.72}{11.66}\selectfont $b$}
	\put(174.97,11.00){\fontsize{9.72}{11.66}\selectfont $d$}
	\put(205.03,52.34){\fontsize{9.72}{11.66}\selectfont $c$}
	\put(18.45,62.12){\fontsize{9.72}{11.66}\selectfont $C_1$}
	\put(80.08,62.12){\fontsize{9.72}{11.66}\selectfont $C_2$}
	\put(6.80,9.50){\fontsize{9.72}{11.66}\selectfont $C_3$}
	\put(220.06,52.34){\fontsize{9.72}{11.66}\selectfont $a$}
	\end{picture}%
	\fi
	\caption{\label{fig:F}%
		Example \ref{example:Brunnian braid}}
\end{figure}
\end{example}

Example \ref{example:communi} is Exercise 8 of  \cite{D.Rolfsen} Chapter 3 section F, also in \cite{TK}. Exercise 7 of \cite{D.Rolfsen} Chapter 3 section F and Fig.1.10 in \cite{BS} can be proved in the same manner.

Example \ref{example:gBr} is Tait series of Fig.2 in \cite{J}, the simplest one of which is Borremean ring. \cite{J} shows lots of \emph{Borromean links} and its Fig.4, 5 and 6, which generalize Borromean ring in different ways, can be proved to be Brunnian similarly. The first 6 links of Fig.7 and the first 6 links of Fig.8 in \cite{J} are more complicated but Arc-criterion works as well.

Example \ref{example:Brunnian braid} is Fig.5 in \cite{F}, constructed from Brunnian braids.

Compared with original methods in these papers, Arc-method is clearly more concise and faster in verifying Brunnian property.
However, there are also cases, for example Brunn's chain (Fig. \ref{fig:milnorlink}), where Arc-method fails in practice. The obstruction is that there may be infinitely many ways to connect two points on a disk with marked points, see Fig. \ref{fig:auxiliaryarc}. So if there are more than two intersection points on the same side of a test disk, there will be infinitely many candidates for incredible arcs of one ear, and Arc-method may get useless in practice. Deleting some intersection points, as used in Example \ref{example:gBr}, may help but not always. To overcome this limitation, another method, Circle-method, develops.

\begin{figure}[htbp]
	\centering
	\ifpdf
	\setlength{\unitlength}{1bp}%
	\begin{picture}(341.28, 81.37)(0,0)
	\put(0,0){\includegraphics{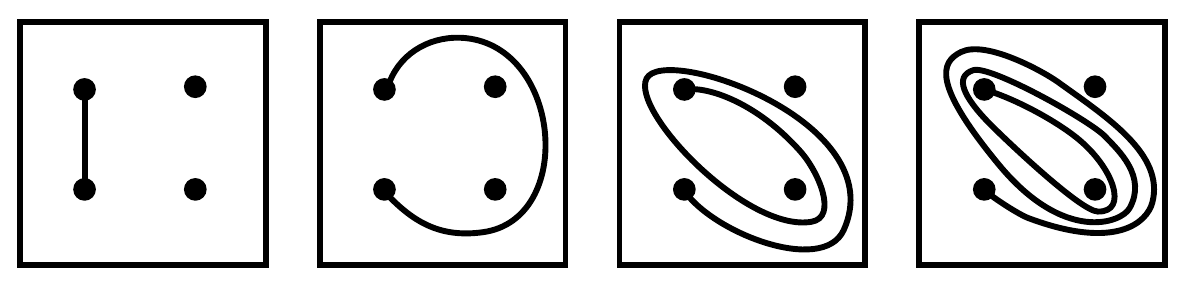}}
	\end{picture}%
	\else
	\setlength{\unitlength}{1bp}%
	\begin{picture}(341.28, 81.37)(0,0)
	\put(0,0){\includegraphics{auxiliaryarc}}
	\end{picture}%
	\fi
	\caption{\label{fig:auxiliaryarc}%
		Different ways to connect two points}
\end{figure}

\section{Circle-method}\label{sect:circle}

\subsection{Quasi-tangles}\label{subsect:terminologyC}\

Suppose an $n$-component link $L=\cup_{i=1}^n C_i$ satisfies (ST) and is divided into sublinks $L_I$, $L_J$, and $L_0$ as in Preliminaries, with respect to test disks $\{D_i\}_{i\in I}$. We have thicken disks $\{\mathbf{D_i}\}_{i\in I}$ and the test complex $U$.

Fix an index $i\in I$ and a test subcomplex $U'$ containing $\mathbf{D_i}$. Recall that $\mathbf{D_i}$ has two sides $D_i^+,D_i^-$, so for each connected component of $U'-\mathbf{D_i}$, say $X$, one of the following four cases holds:
\begin{center}
	Case $+$: $\overline{X}\cap D_i^+\neq\emptyset,\,\overline{X}\cap D_i^-=\emptyset$;
	
	Case $-$: $\overline{X}\cap D_i^+=\emptyset,\,\overline{X}\cap D_i^-\neq\emptyset$;
	
	Case $\pm$: $\overline{X}\cap D_i^+\neq\emptyset,\,\overline{X}\cap D_i^-\neq\emptyset$;
	
	Case $0$: $\overline{X}\cap D_i^+=\emptyset,\,\overline{X}\cap D_i^-=\emptyset$,
\end{center}
where $\overline{X}$ represents the closure of $X$ in $\mathbb{R}^3$. Denote the sets consisting of such $X$'s in the four cases by $K_i^+,K_i^-,K_i^\pm,K_i^0$ respectively. For any subset $Q$ of them, we always denote $\bigcup_{X\in Q}\overline{X}$ by $|Q|$.

\begin{defn}\label{def:quasitangle}
	Let $\sigma$ be $+$ or $-$. For $U'$, a subset $Q \subset K_i^\sigma\cup K_i^0$ is a \emph{quasi-tangle} on $D_i^\sigma$, if one of the following two conditions holds:

	\begin{enumerate}
		\item $Q \cap K_i^\sigma \neq \emptyset$ and $L\cap D_i^\sigma - |Q| \neq \emptyset$;
		\item $Q$ only contains an ear on $D_i^\sigma$.
	\end{enumerate}
\end{defn}

\begin{rem}\label{rmk}
	We require $L\cap D_i^\sigma - |Q| \neq \emptyset$ in case that we, in practice, get caught in the task (see Flowchart \ref{fig:flowchart}) to prove $C_i$ itself can not bound an open disk in the complement of $U$, which is exactly the origin problem by Proposition \ref{prop:basic} and makes our method meaningless.
\end{rem}

\begin{figure}[htbp]
	\centering
	\ifpdf
	\setlength{\unitlength}{1bp}%
	\begin{picture}(315.36, 155.75)(0,0)
	\put(0,0){\includegraphics{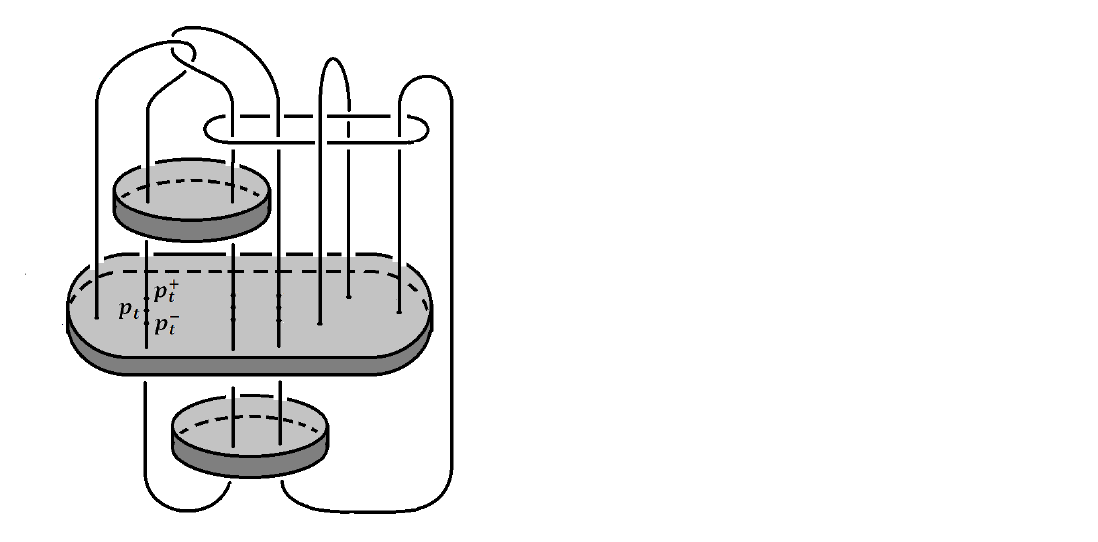}}
	\put(20.63,97.12){\fontsize{10.69}{12.83}\selectfont $a$}
	\put(119.72,106.13){\fontsize{10.69}{12.83}\selectfont $b$}
	\put(76.86,138.60){\fontsize{10.69}{12.83}\selectfont $c$}
	\put(102.55,97.12){\fontsize{10.69}{12.83}\selectfont $e$}
	\put(33.20,31.96){\fontsize{10.69}{12.83}\selectfont $d$}
	\put(122.03,31.96){\fontsize{10.69}{12.83}\selectfont $f$}
	\put(8.06,78.05){\fontsize{8.55}{10.26}\selectfont $D_1^+$}
	\put(8.06,47.89){\fontsize{8.55}{10.26}\selectfont $D_1^-$}
	\put(5.67,62.34){\fontsize{8.55}{10.26}\selectfont $\mathbf{D_1}$}
	\put(29.01,109.90){\fontsize{8.55}{10.26}\selectfont $\mathbf{D_2}$}
	\put(66.72,9.63){\fontsize{8.55}{10.26}\selectfont $\mathbf{D_3}$}
	\put(143.71,136.20){\fontsize{10.69}{12.83}\selectfont Thicken disks: $\mathbf{D_1,D_2,D_3}$.}
	\put(143.71,66.61){\fontsize{10.69}{12.83}\selectfont Ear on $D_1^+$: $e$.}
	\put(143.71,122.54){\fontsize{10.69}{12.83}\selectfont $K_1^+=\{ a\cup c\cup \mathbf{D_2},e \}$,}
	\put(143.71,108.45){\fontsize{10.69}{12.83}\selectfont $K_1^-=\emptyset$,}
	\put(143.71,94.35){\fontsize{10.69}{12.83}\selectfont $K_1^ \pm=\{ d\cup f\cup \mathbf{D_3} \}$,}
	\put(143.71,80.26){\fontsize{10.69}{12.83}\selectfont $K_1^ 0=\{  b \}$.}
	\put(143.71,53.27){\fontsize{10.69}{12.83}\selectfont Quasi-tangles on $D_1^+$: }
	\put(150.75,39.51){\fontsize{10.69}{12.83}\selectfont $\{ a\cup c\cup \mathbf{D_2} \}, \{ e \},$}
	\put(150.75,25.42){\fontsize{10.69}{12.83}\selectfont $\{ a\cup c\cup \mathbf{D_2}, e \}, \{a\cup c\cup \mathbf{D_2}, b \}, \{ e, b \}, $}
	\put(150.75,11.82){\fontsize{10.69}{12.83}\selectfont $\{ a\cup c\cup \mathbf{D_2}, e, b \}.$}
	\end{picture}%
	\else
	\setlength{\unitlength}{1bp}%
	\begin{picture}(315.36, 155.75)(0,0)
	\put(0,0){\includegraphics{test-subcomplex}}
	\put(20.63,97.12){\fontsize{10.69}{12.83}\selectfont $a$}
	\put(119.72,106.13){\fontsize{10.69}{12.83}\selectfont $b$}
	\put(76.86,138.60){\fontsize{10.69}{12.83}\selectfont $c$}
	\put(102.55,97.12){\fontsize{10.69}{12.83}\selectfont $e$}
	\put(33.20,31.96){\fontsize{10.69}{12.83}\selectfont $d$}
	\put(122.03,31.96){\fontsize{10.69}{12.83}\selectfont $f$}
	\put(8.06,78.05){\fontsize{8.55}{10.26}\selectfont $D_1^+$}
	\put(8.06,47.89){\fontsize{8.55}{10.26}\selectfont $D_1^-$}
	\put(5.67,62.34){\fontsize{8.55}{10.26}\selectfont $\mathbf{D_1}$}
	\put(29.01,109.90){\fontsize{8.55}{10.26}\selectfont $\mathbf{D_2}$}
	\put(66.72,9.63){\fontsize{8.55}{10.26}\selectfont $\mathbf{D_3}$}
	\put(143.71,136.20){\fontsize{10.69}{12.83}\selectfont Thicken disks: $\mathbf{D_1,D_2,D_3}$.}
	\put(143.71,66.61){\fontsize{10.69}{12.83}\selectfont Ear on $D_1^+$: $e$.}
	\put(143.71,122.54){\fontsize{10.69}{12.83}\selectfont $K_1^+=\{ a\cup c\cup \mathbf{D_2},e \}$,}
	\put(143.71,108.45){\fontsize{10.69}{12.83}\selectfont $K_1^-=\emptyset$,}
	\put(143.71,94.35){\fontsize{10.69}{12.83}\selectfont $K_1^ \pm=\{ d\cup f\cup \mathbf{D_3} \}$,}
	\put(143.71,80.26){\fontsize{10.69}{12.83}\selectfont $K_1^ 0=\{  b \}$.}
	\put(143.71,53.27){\fontsize{10.69}{12.83}\selectfont Quasi-tangles on $D_1^+$: }
	\put(150.75,39.51){\fontsize{10.69}{12.83}\selectfont $\{ a\cup c\cup \mathbf{D_2} \}, \{ e \},$}
	\put(150.75,25.42){\fontsize{10.69}{12.83}\selectfont $\{ a\cup c\cup \mathbf{D_2}, e \}, \{a\cup c\cup \mathbf{D_2}, b \}, \{ e, b \}, $}
	\put(150.75,11.82){\fontsize{10.69}{12.83}\selectfont $\{ a\cup c\cup \mathbf{D_2}, e, b \}.$}
	\end{picture}%
	\fi
	\caption{\label{fig:testcomplex}%
		A test subcomplex and its quasi-tangles}
\end{figure}

Let $Q$ be a quasi-tangle on $D_i^\sigma$ for $U'$. If there exists a circle $C \subset D_i^\sigma$ bounding disks $D_C,D_Q$, such that $D_C\subset D_i^\sigma,{\rm Int}(D_Q)\subset\mathbb{R}^3-U'$ and $D_C \cup D_Q$ bounds a 3-ball $B$ with $\overline{{\rm Int}(B)\cap U'} = |Q|$, then we say $Q$ is \emph{compressible for $U'$}. The circle $C$ is called an \emph{incredible circle of $Q$ for $U'$}, and $D_Q$ is called a \emph{compressing disk for $U'$}.

We point out that, if a quasi-tangle $Q$ is compressible for $U$ and $|Q \cap K_i^+ |$ is an ear $e$ of $C_j$, then $e$ is unknotted. In fact, otherwise $e$ would be a connected summand of the unknot $C_j$, which is impossible. As a consequence, it is easy to see that an ear $e$ is compressible for $U$ if and only if the quasi-tangle $\{e\}$ is compressible for $U$.

\subsection{Criterion}\label{subsect:criterionC}\

\begin{criterion}\label{criterionC}(\textbf{Circle-criterion})
	Let $L=\cup_{i=1}^n C_i$ be an $n$-component link with property (ST). Suppose $L$ is divided into sublinks $L_I$, $L_J$, and $L_0$ with respect to test disks $\{D_i\}_{i\in I}$ as in Preliminaries and $U$ is the test complex. If each quasi-tangle is not compressible for $U$, then $L$ is Brunnian.
\end{criterion}

\begin{proof}
	Assume to the contrary that $L$ is trivial. By Proposition \ref{prop:basic}, there are disjoint disks $\sqcup_{j \in J} D_j\subset \mathbb{R}^3 - L_I-L_0$ bounded by $L_J$ with $C_j =\partial D_j$. Consider the intersection $S\triangleq\cup_{j \in J} D_j\cap\cup_{i\in I}D_i$. It is a disjoint union of finitely many circles and arcs, and the endpoints of arc components are all on $L_J$. Any circle component $C\subset D_i\cap S$ $(i\in I)$ bounds a disk $D_i ^C$ on $D_i$. Denote $S_C\triangleq D_i ^C\cap S$.
	
	We can modify $\{D_j\}_{j\in J}$ so that $S_C$ includes at least one arc component for every circle component $C\subset S$. In fact, if $S_C$ has no arc component, then there is a circle component $C'\subset S_C$ innermost on $D_i$, that is, $S_{C'}=C'$. Thus we can replace the disk $C'$ bounds on that $D_j$ by $D_i ^{C'}$ and lift or lower it a little to eliminate $C'$ (see Fig. \ref{fig:innermost}). Inductively we get what we want.
	
	\begin{figure}[htbp]
		\centering
		\includegraphics[height=2cm]{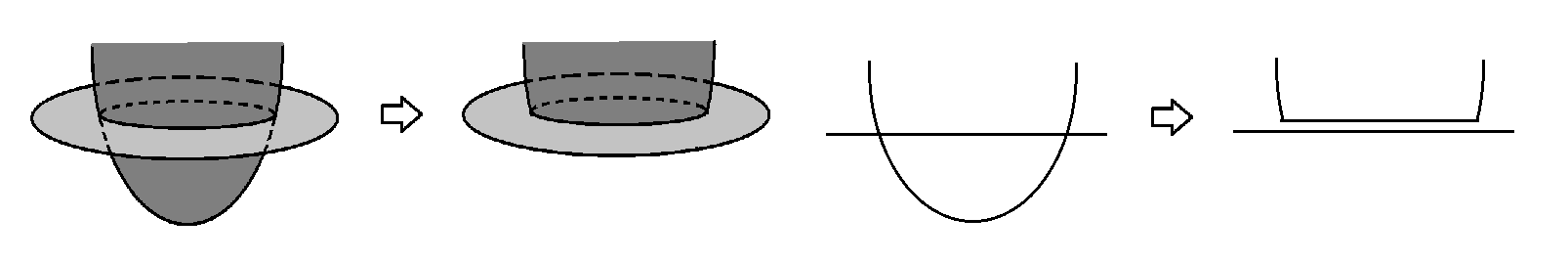}
		\caption{Eliminating an innermost circle}
		\label{fig:innermost}
	\end{figure}
	
	Moreover, modify $\{D_j\}_{j\in J}$ so that any circle component $C\subset S\cap D_i$ $(i\in I)$ is not outermost, that is, $S\cap D_i -S_C \neq \emptyset$. Otherwise $C\cup C_i$ bounds an annulus $A$, so $C$ can be eliminated by replacing a tubular neighbourhood of $C$ on $D_j$ by an annulus surrounding nearly outside $A$, as in Fig. \ref{fig:outermost}. Now fix such disks $\{D_j\}_{j\in J}$ and still denote $\cup_{j \in J} D_j\cap\cup_{i\in I}D_i$ by $S$.

	\begin{figure}[htbp]
		\centering
		\includegraphics[height=1.7cm]{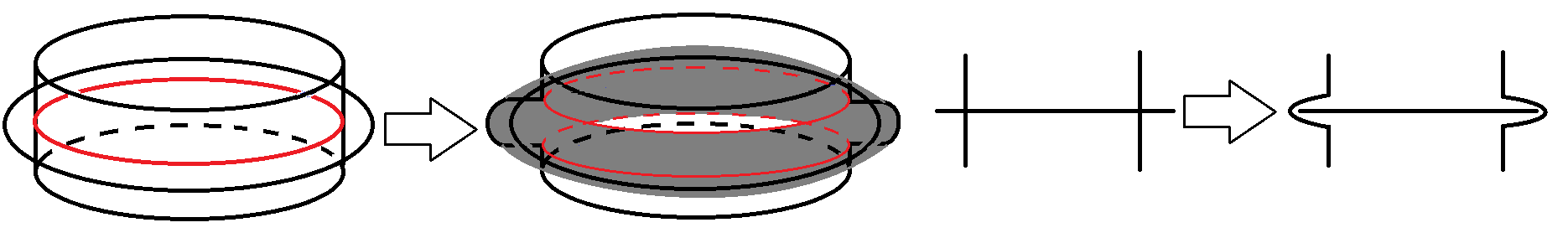}
		\caption{Eliminating an outermost circle}
		\label{fig:outermost}
	\end{figure}
	
	For any $j\in J$, all arc components of $D_j\cap S$ cut $D_j$ into some regions, each isomorphic to a disk. There are at least 2 such regions, say $R_1,R_2$, outermost on $D_j$, i.e., they include no arc component of $S$. If $R_1\cap S=\emptyset$, then $\partial R_1 \cap C_j$ is an ear compressible for $U$ with compressing disk $R_1$. Otherwise, there exists a circle $C\subset R_1\cap S$ innermost on $D_j$. Suppose $C\subset D_i$ $(i\in I)$ bounds $D_i ^C,D_j ^C$ on $D_i,D_j$ respectively, and $B$ is the 3-ball bounded by the sphere $D_i ^C\cup D_j ^C$. As $S_C$ includes an arc component of $S$, $B$ contains a quasi-tangle which is thus compressible for $U$. In either case, there is a contradiction.
\end{proof}

This criterion, as well as Arc-criterion, provides a necessary and sufficient condition for $L$ to be trivial. In fact, if there is a compressible quasi-tangle on $D_i^+$ or $D_i^-$, we can compress it across the thicken disk $\mathbf{D_i}$ by an isotopy of $L$, to eliminate it. Do it inductively. $L$ is trivial if and only if finally all intersections are eliminated.

Generally, let $L$ be a link with each component unknotted. Suppose each proper sublink of $L$ is trivial, otherwise consider its proper sublinks first. Then we can follow Flowchart  \ref{fig:flowchart} to decide whether it is Brunnian. We consider quasi-tangles of cardinality $1,2,\cdots$ in order so that Sub-criterion \ref{prop:subcomplex} and \ref{prop:boundaryconnectedsum} can be applied to reduce the task.

\begin{figure}
	\centering
	\tikzstyle{startstop} =[rectangle,rounded corners, minimum width = 2cm, minimum height=1cm,text centered, draw = red, line width=1pt]
	\tikzstyle{io} = [trapezium, trapezium left angle=70, trapezium right angle=110, minimum width=1cm, minimum height=1cm, text centered, draw=black, inner sep=12pt]
	\tikzstyle{process} = [rectangle, minimum width=2cm, minimum height=1cm, text centered, draw=black]
	\tikzstyle{decision} = [diamond, aspect = 4, text centered, draw=black]
	\tikzstyle{arrow} = [->,>=stealth]
	
	\begin{tikzpicture}[node distance=2cm]
	
	\node[startstop](start){Start};
	\node[io, below of = start, yshift = 0.5cm](link){Link $L=\cup_{i=1}^n C_i$ satisfying (ST)};
	\node[process, align=center, below of = link, yshift = 0.5cm](test){Fix test disks $\{D_i\}_{i\in I}$ and test \\ complex $U = \coprod_{i \in I} \mathbf{D_i} \cup L$.};
	\node[decision, below of = test, yshift = 0.3cm](noLJ){$L_J$ is empty?};
	\node[process, below of = noLJ, xshift=0cm, yshift = 0cm](sets1){$s=1$};
	\node[io, left of = sets1, node distance = 5cm](trivial){$L$ is trivial.};
	\node[process, below of = sets1, yshift = 0cm](qt){Consider all quasi-tangles of cardinality $s$.};
	\node[startstop, left of = qt, node distance= 5cm, yshift = 0cm](stop){Stop};
	\node[decision, below of = qt, yshift = 0cm](exist){Do they exist?};
	\node[process, right of = exist, node distance=4cm](sets2){$s=s+1$};
	\node[decision, below of = exist, yshift=0cm](compressible){All incompressible for $U$?};
	\node[process, align=center, below of = compressible, yshift = 0cm](eliminate){Eliminate a compressible quasi-tangle \\ by pushing it through the corresponding thicken disk; \\ equivalently, adjust the test disks and test complex.};
	\node[io, left of = exist, xshift=-3cm, yshift = 0cm](Brunnian){$L$ is Brunnian.};
	\coordinate (point) at (5.5cm, -3cm);
	
	\draw [arrow] (start) -- (link);
	\draw [arrow] (link) -- (test);
	\draw [arrow] (test) -- (noLJ);
	\draw [arrow] (noLJ) -- node[right]{No}(sets1);
	\draw [arrow] (noLJ) -| node[above]{Yes}(trivial);
	\draw [arrow] (trivial) -- (stop);
	\draw [arrow] (sets1) -- (qt);
	\draw [arrow] (qt) -- (exist);
	\draw [arrow] (exist) -- node[above] {No} (Brunnian);
	\draw [arrow] (Brunnian) -- (stop);
	\draw [arrow] (exist) -- node[right] {Yes} (compressible);
	\draw [arrow] (compressible) -| node[below] {Yes} (sets2);
	\draw [arrow] (sets2) |- (qt);
	\draw [arrow] (compressible) -- node[right] {No} (eliminate);
	\draw (eliminate) -| (point);
	\draw [arrow] (point) -- (test);
	\end{tikzpicture}
	\caption{Flowchart of Circle-method}
	\label{fig:flowchart}
\end{figure}
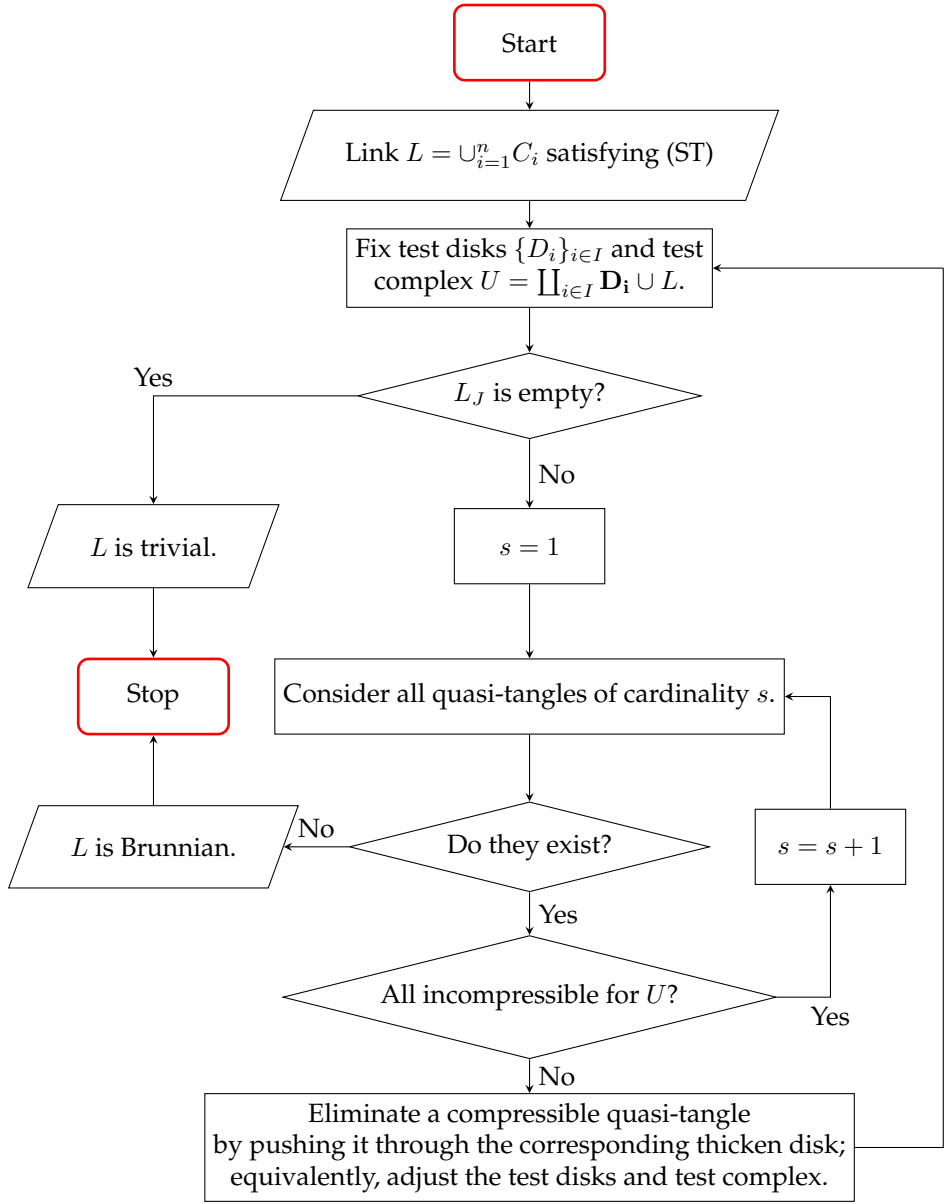

\begin{rem}
	There may be only one compressible quasi-tangle for an unlink with respect to some test complex. For instance, see Fig. \ref{fig:uniqueqt}, where the test disk is colored grey and the only compressible quasi-tangle is in the red box.
	
	\begin{figure}[htbp]
		\centering
		\includegraphics[height=5cm]{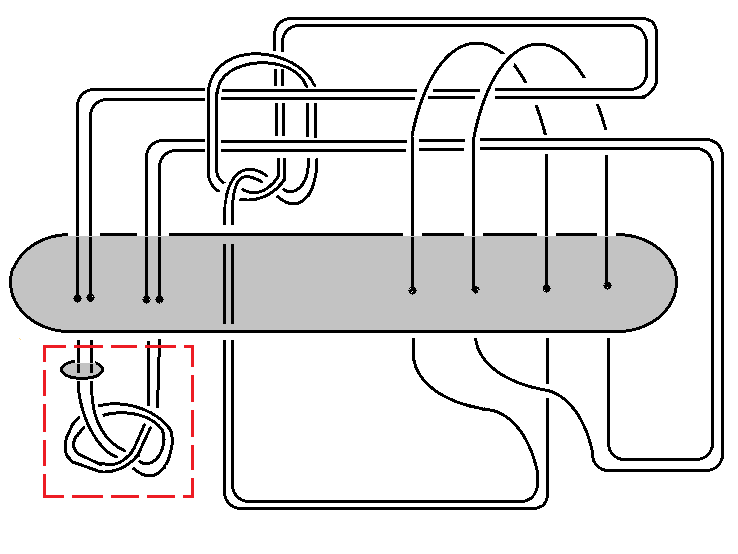}
		\caption{An unlink with only one compressible quasi-tangle}
		\label{fig:uniqueqt}
	\end{figure}
\end{rem}

\subsection{Auxiliary criteria}\label{subsect:auxiliary}\

Given a link $L$ with (ST), to prove (NT) by Circle-criterion, we need to choose some text complex $U$ and show all quasi-tangles are incompressible for $U$. It is sometimes no easy task. We provide four helpful criteria here. The first one can be easily proved by contradiction, and proofs of the others will be given in Subsection \ref{subsect:proof}.

\begin{subcriterion}\label{prop:subcomplex}(\textbf{Containment-criterion})
	Let $U'_1$ and $U'_2$ be test subcomplexes with $U'_1 \subset U'_2$. If a quasi-tangle $Q_1$ is incompressible for $U'_1$, then for any quasi-tangle $Q_2$ with $|Q_1|\subset |Q_2|\subset U'_2$ and $|Q_2-Q_1| \subset U'_2 - U'_1$, $Q_2$ is incompressible for $U'_2$.
\end{subcriterion}

Fix a test subcomplex $U'$ and without loss of generality, suppose $Q$ is a quasi-tangle on $D_1^+$. Take a disk $D_0\subset D_1^+$ such that $(D_1^+ - D_0)\cap L =\emptyset$ and consider the test subcomplex $U'_Q  \triangleq \mathbf{D_1} \cup |Q|$. Then $\partial D_0$ bounds a disk $D$ outside $U'_Q$ such that $D_0 \cup D$ bounds a 3-ball $B$ including $|Q|$. Lemma \ref{lem:tangleball} in Subsection \ref{subsect:proof} shows $B$ is unique for $Q$ in some sense and thus we also call the pair $(B,Q)$ a quasi-tangle.

\begin{defn}\label{def:unsplit}
	A quasi-tangle $(B,Q)$ is \emph{split}, if there exists a properly embedded disk $D$ in $B$ avoiding $|Q|$,  and a partition $Q=Q_1\cup Q_2$, such that $\partial D$ intersects $\partial D_0$ with two points, and $Q_1,Q_2$ are both non-empty and included in the different components of $B-D$. (Here $Q_1$ or $Q_2$ is not necessary to be a quasi-tangle.) $(B,Q)$ is \emph{unsplit} if it is not split.
\end{defn}

\begin{subcriterion}\label{prop:boundaryconnectedsum}(\textbf{Split-criterion})
	Let $U'$ be a test subcomplex and $Q$ be a quasi-tangle for $U'$. Suppose $Q$ is split with respect to the partition $Q=Q_1\cup Q_2$. If $Q_1$ is an incompressible quasi-tangle for $U'$, then $Q$ is also incompressible.
\end{subcriterion}

Therefore, in Criterion \ref{criterionC}, it suffices to show every unsplit quasi-tangle is incompressible for $U$.

Next we give an equivalent condition for a quasi-tangle to be compressible. Fix a test subcomplex $U'$ and consider a quasi-tangle on $D_1^+$, say $(B,Q)$. See Fig. \ref{fig:trivializedsubstitution}. Take a 3-ball $B_T$ in the closure of $\mathbb{R}^3-U'$ such that $B_T \cap U'$ is a disk $D_T\subset D_1^+$. There is an orientation-preserving homeomorphism $h:B\to B_T$ that maps $B\cap D_1^+$ to $D_T$. The complex $\overline{U'-|Q|} \cup h(|Q|)$ is called a \emph{trivial substitution of $Q$ for $U'$}. Up to isotopy, it is unique and independent of the choices of $h,B,B_T$.

\begin{figure}[htbp]
	\centering
	\includegraphics[height=2.5cm]{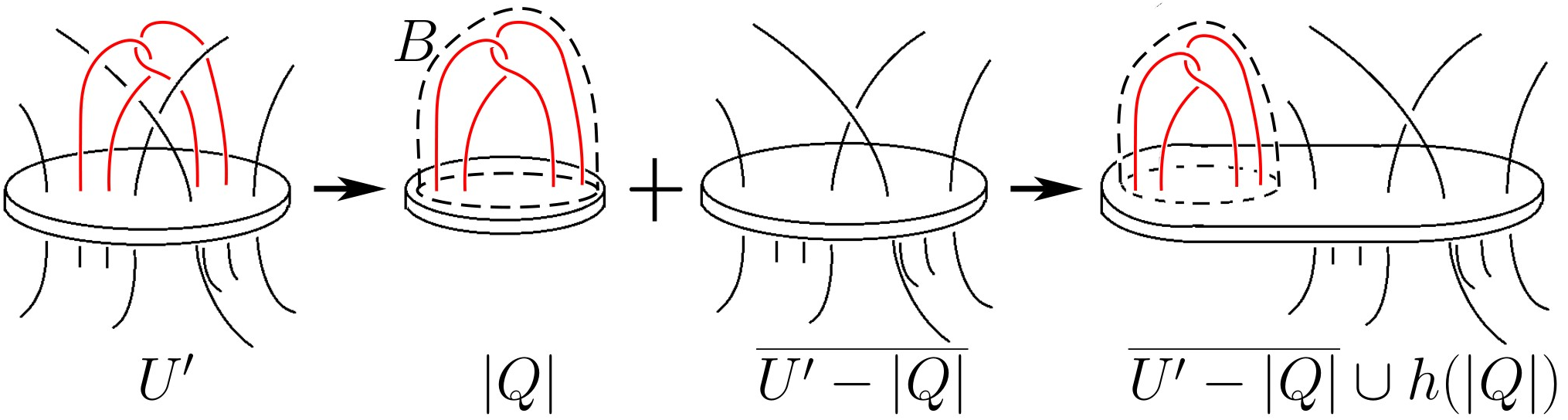}
	\caption{Trivial substitution}
	\label{fig:trivializedsubstitution}
\end{figure}

\begin{subcriterion}\label{cor:incompressibility}(\textbf{Substitution-criterion})
	Let $Q$ be a quasi-tangle for test subcomplex $U'$. If the trivial substitution of $Q$ for $U'$ is not isotopic to $U'$ in $\mathbb{R}^3$, then $Q$ is incompressible for $U'$.
\end{subcriterion}

Instead of investigating each unsplit quasi-tangle, we can also consider all candidates for incredible circles. If we can somehow exclude their potential, then by Circle-criterion, (NT) is proved. Generally there are infinitely many such candidate circles, and we need to narrow our focus.

\begin{subcriterion}\label{prop:incrediblecircle}(\textbf{Criterion for incredible circles})
	Let $C$ be an incredible circle of an unsplit quasi-tangle $Q$ for the test complex $U$. Fix a test subcomplex $U'$ containing $|Q|$ and suppose $C'$ is an incredible circle of $Q$ for $U'$. If there is no compressible quasi-tangle on $D_1^+$ for the test subcomplex $\overline{U'-|Q|}$, then $C$ is isotopic to $C'$ on $D_1^+$ in $U'$.
\end{subcriterion}

\subsection{Example analysis}\label{subsect:examplesC}\

We have established Circle-method, consisting of Circle-criterion, a flowchart, and four sub-criteria, in the previous two subsections. Some typical examples demonstrate efficiency of the criteria.

\begin{example}\label{example:gWl}
	(Fig. \ref{fig:gWl}) Take the test disk colored in gray and test subcomplex $U'$ on the right. There are two unknotted ears $e_1,e_2$, and only two quasi-tangles $\{e_1\},\{e_2\}$. Connect the endpoints of the ears by arcs $\alpha_1,\alpha_2$ on the corresponding sides, then we see $lk ( e_1 \cup \alpha_1 , e_2 \cup \alpha_2 ) =1$, thus neither quasi-tangle is compressible for $U'$. By Sub-criterion \ref{prop:subcomplex} and Circle-criterion, the link is Brunnian.
	
	\begin{figure}[htbp]
		\centering
		\ifpdf
		\setlength{\unitlength}{1bp}%
		\begin{picture}(140.49, 93.04)(0,0)
		\put(0,0){\includegraphics{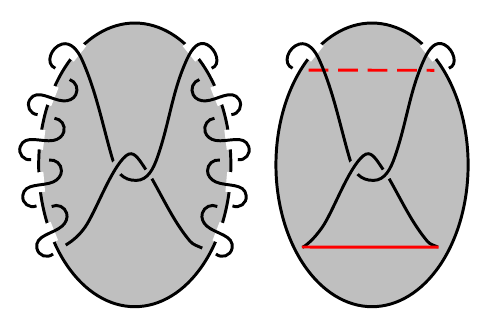}}
		\put(88.77,58.15){\fontsize{8.03}{9.64}\selectfont $e_1$}
		\put(118.90,33.30){\fontsize{8.03}{9.64}\selectfont $e_2$}
		\put(103.06,76.16){\fontsize{8.03}{9.64}\selectfont $\alpha_1$}
		\put(103.06,15.91){\fontsize{8.03}{9.64}\selectfont $\alpha_2$}
		\end{picture}%
		\else
		\setlength{\unitlength}{1bp}%
		\begin{picture}(140.49, 93.04)(0,0)
		\put(0,0){\includegraphics{gWl}}
		\put(88.77,58.15){\fontsize{8.03}{9.64}\selectfont $e_1$}
		\put(118.90,33.30){\fontsize{8.03}{9.64}\selectfont $e_2$}
		\put(103.06,76.16){\fontsize{8.03}{9.64}\selectfont $\alpha_1$}
		\put(103.06,15.91){\fontsize{8.03}{9.64}\selectfont $\alpha_2$}
		\end{picture}%
		\fi
		\caption{\label{fig:gWl}%
			Generalized Whitehead link}
	\end{figure}
\end{example}

The third, fourth and last picture in \cite{D.Rolfsen} Chapter 5 section E Exercise 5 generalize Whitehead link in different ways. The fourth is Example \ref{example:gWl} and the other two can be proved by Arc-criterion. The link occurs in Chapter 7 section C Theorem 5 in \cite{D.Rolfsen}, and Fig. 1.5 in \cite{BS} are also Brunnian, proved similarly by Circle-criterion.

\begin{example} \label{example:Milnor}
	Milnor link and Brunn's chain, as shown in Fig. \ref{fig:milnorlink}.
	
	For Milnor link, take the gray test disk and by Arc-method, we only need to show the ear $e$ is incompressible. Connecting the two endpoints by arc $\alpha$, we obtain another Milnor link with one less component and it suffices to prove the new link is Brunnian. In this way we can begin recursion until Hopf link.
	
	For Brunn's chain, take the test disk colored in gray and fix sublinks $L_I,L_J,L_0$ as in Section \ref{sect:preliminaries}. The unknotted ears $e_1,e_3$ are symmetric, and so are $e_2,e_4$. Delete $e_3$ and $e_4$, and connect the endpoints of $e_1,e_2$ by arcs $\alpha_1,\alpha_2$ on corresponding sides respectively, then $e_1 \cup \alpha_1,e_2 \cup \alpha_2$ and $L_3$ form a Milnor link. So neither $e_1$ nor $e_2$ is compressible. According to the definition, each quasi-tangle contains only one of $e_1,e_2,e_3,e_4$ and some components of $L_3$. One may check easily that such a quasi-tangle is split. So by Sub-criterion \ref{prop:boundaryconnectedsum}, every quasi-tangle is incompressible.
	
	\begin{figure}[htbp]
		\centering
		\ifpdf
		\setlength{\unitlength}{1bp}%
		\begin{picture}(325.65, 185.42)(0,0)
		\put(0,0){\includegraphics{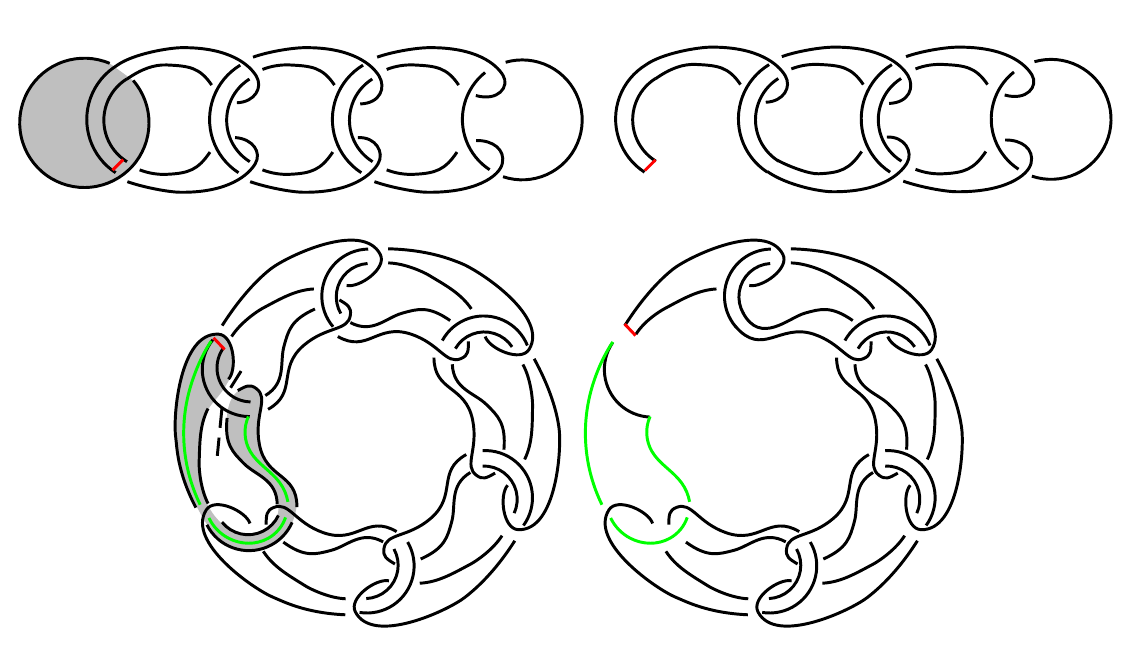}}
		\put(64.41,104.32){\fontsize{8.83}{10.60}\selectfont $e_1$}
		\put(84.14,72.55){\fontsize{8.83}{10.60}\selectfont $e_3$}
		\put(181.11,104.32){\fontsize{8.83}{10.60}\selectfont $e_1$}
		\put(179.10,70.88){\fontsize{8.83}{10.60}\selectfont $e_2$}
		\put(181.11,84.26){\fontsize{8.83}{10.60}\selectfont \textcolor[rgb]{1, 0, 0}{$\alpha_1$}}
		\put(194.48,50.82){\fontsize{8.83}{10.60}\selectfont \textcolor[rgb]{0, 1, 0}{$\alpha_2$}}
		\put(69.76,79.41){\fontsize{8.83}{10.60}\selectfont $e_4$}
		\put(59.73,49.31){\fontsize{8.83}{10.60}\selectfont $e_2$}
		\put(41.54,172.85){\fontsize{8.83}{10.60}\selectfont $e$}
		\put(188.83,157.84){\fontsize{8.83}{10.60}\selectfont $e$}
		\put(186.95,132.65){\fontsize{8.83}{10.60}\selectfont \textcolor[rgb]{1, 0, 0}{$\alpha$}}
		\end{picture}%
		\else
		\setlength{\unitlength}{1bp}%
		\begin{picture}(325.65, 185.42)(0,0)
		\put(0,0){\includegraphics{Milnorlink}}
		\put(64.41,104.32){\fontsize{8.83}{10.60}\selectfont $e_1$}
		\put(84.14,72.55){\fontsize{8.83}{10.60}\selectfont $e_3$}
		\put(181.11,104.32){\fontsize{8.83}{10.60}\selectfont $e_1$}
		\put(179.10,70.88){\fontsize{8.83}{10.60}\selectfont $e_2$}
		\put(181.11,84.26){\fontsize{8.83}{10.60}\selectfont \textcolor[rgb]{1, 0, 0}{$\alpha_1$}}
		\put(194.48,50.82){\fontsize{8.83}{10.60}\selectfont \textcolor[rgb]{0, 1, 0}{$\alpha_2$}}
		\put(69.76,79.41){\fontsize{8.83}{10.60}\selectfont $e_4$}
		\put(59.73,49.31){\fontsize{8.83}{10.60}\selectfont $e_2$}
		\put(41.54,172.85){\fontsize{8.83}{10.60}\selectfont $e$}
		\put(188.83,157.84){\fontsize{8.83}{10.60}\selectfont $e$}
		\put(186.95,132.65){\fontsize{8.83}{10.60}\selectfont \textcolor[rgb]{1, 0, 0}{$\alpha$}}
		\end{picture}%
		\fi
		\caption{\label{fig:milnorlink}%
			Milnor link and Brunn's chain}
	\end{figure}
\end{example}

Milnor link occurs in Fig. 7 in \cite{M}, Fig. 2.25 in \cite{BS}, and  Brunn's chain appears in \cite{B,HDb}, also in \cite{D.Rolfsen} Chapter 7 section J, and Fig. 2.22 and 2.28 in \cite{BS}.

The quick proofs above illustrate how Sub-criterion \ref{prop:subcomplex} and \ref{prop:boundaryconnectedsum} simplify the task in Flowchart \ref{fig:flowchart}, while the next two examples show the way Sub-criterion \ref{cor:incompressibility} and \ref{prop:incrediblecircle} work for complicated links.
	
\begin{figure}[htbp]
	\centering
	\ifpdf
	\setlength{\unitlength}{1bp}%
	\begin{picture}(350.85, 82.56)(0,0)
	\put(0,0){\includegraphics{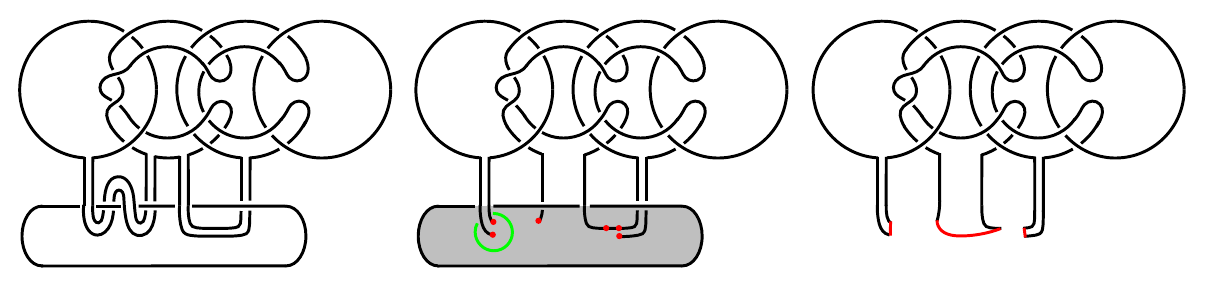}}
	\put(89.40,7.97){\fontsize{10.69}{12.83}\selectfont $C_1$}
	\put(7.89,51.44){\fontsize{10.69}{12.83}\selectfont $C_2$}
	\put(98.91,51.44){\fontsize{10.69}{12.83}\selectfont $C_3$}
	\put(236.12,51.44){\fontsize{10.69}{12.83}\selectfont $C$}
	\put(126.46,31.61){\fontsize{10.69}{12.83}\selectfont $e_1$}
	\put(188.41,31.61){\fontsize{10.69}{12.83}\selectfont $e_2$}
	\put(327.42,51.44){\fontsize{10.69}{12.83}\selectfont $C_3$}
	\end{picture}%
	\else
	\setlength{\unitlength}{1bp}%
	\begin{picture}(350.85, 82.56)(0,0)
	\put(0,0){\includegraphics{Lf1}}
	\put(89.40,7.97){\fontsize{10.69}{12.83}\selectfont $C_1$}
	\put(7.89,51.44){\fontsize{10.69}{12.83}\selectfont $C_2$}
	\put(98.91,51.44){\fontsize{10.69}{12.83}\selectfont $C_3$}
	\put(236.12,51.44){\fontsize{10.69}{12.83}\selectfont $C$}
	\put(126.46,31.61){\fontsize{10.69}{12.83}\selectfont $e_1$}
	\put(188.41,31.61){\fontsize{10.69}{12.83}\selectfont $e_2$}
	\put(327.42,51.44){\fontsize{10.69}{12.83}\selectfont $C_3$}
	\end{picture}%
	\fi
	\caption{\label{fig:Lf1}%
		Example \ref{example:Lf1}}
\end{figure}
	
\begin{example}\label{example:Lf1} (Fig. \ref{fig:Lf1})
	Choose the test subcomplex $U'$ shown in the middle. There are only two ears $e_1$ and $e_2$ on different sides of the test disk, and by Sub-criterion \ref{prop:boundaryconnectedsum} we only need to consider these two ears. In the complement of the trivial substitution of $\{e_2\}$ for $U'$, $C_3$ can bound a disk. However, there is a nontrivial link as a subset of $U'$, shown on the right, which is similar to Milnor link and implies $C_4$ can not bound a disk in the complement of $U'$. Thus by Sub-criterion \ref{cor:incompressibility}, $\{e_2\}$ is incompressible for $U'$. As for $\{e_1\}$, we see the green circle is the unique candidate for its incredible circle up to isotopy according to Sub-criterion \ref{prop:incrediblecircle}. So it suffices to show the circle $C$ can not bound a disk in the complement of $U'$, which is again implied by the nontrivial link on the right.
\end{example}

Fig. 1 and 2 in \cite{L} are just similar to Example \ref{example:Lf1}.

\begin{example}\label{example:An} (Fig. \ref{fig:A5})
	The link $L_F$ in \cite{HDb}, also $A_n$ of Exercise 15 in Chapter 3 section F in \cite{D.Rolfsen}. Proof for the case $n=2, 3$ need only to be changed a little.
	
	\begin{figure}[htbp]
		\centering
		\ifpdf
		\setlength{\unitlength}{1bp}%
		\begin{picture}(347.58, 124.22)(0,0)
		\put(0,0){\includegraphics{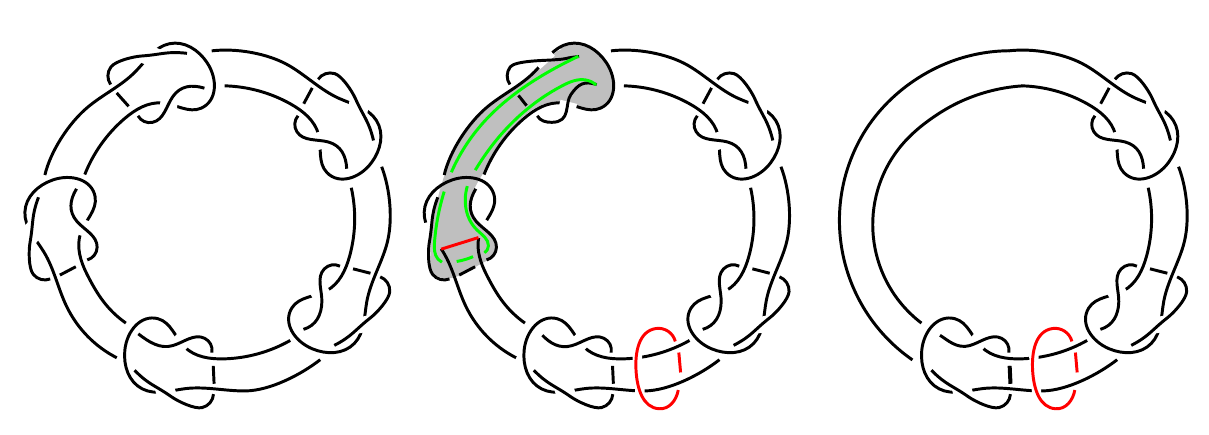}}
		\put(9.76,91.65){\fontsize{10.69}{12.83}\selectfont $C_1$}
		\put(9.76,24.16){\fontsize{10.69}{12.83}\selectfont $C_2$}
		\put(83.66,7.97){\fontsize{10.69}{12.83}\selectfont $C_3$}
		\put(88.72,57.91){\fontsize{10.69}{12.83}\selectfont $C_4$}
		\put(77.25,109.87){\fontsize{10.69}{12.83}\selectfont $C_5$}
		\put(129.06,25.74){\fontsize{10.69}{12.83}\selectfont $e_1$}
		\put(143.14,110.20){\fontsize{10.69}{12.83}\selectfont $e_2$}
		\put(179.09,31.83){\fontsize{10.69}{12.83}\selectfont \textcolor[rgb]{1, 0, 0}{$C$}}
		\put(58.16,57.86){\fontsize{10.69}{12.83}\selectfont $A_5$}
		\put(287.60,57.86){\fontsize{10.69}{12.83}\selectfont $A_3$}
		\put(144.69,67.04){\fontsize{10.69}{12.83}\selectfont $e_4$}
		\put(182.26,90.48){\fontsize{10.69}{12.83}\selectfont $e_3$}
		\put(293.31,31.83){\fontsize{10.69}{12.83}\selectfont \textcolor[rgb]{1, 0, 0}{$C$}}
		\put(126.68,86.53){\fontsize{10.69}{12.83}\selectfont \textcolor[rgb]{0, 1, 0}{$\alpha$}}
		\end{picture}%
		\else
		\setlength{\unitlength}{1bp}%
		\begin{picture}(347.58, 124.22)(0,0)
		\put(0,0){\includegraphics{A5}}
		\put(9.76,91.65){\fontsize{10.69}{12.83}\selectfont $C_1$}
		\put(9.76,24.16){\fontsize{10.69}{12.83}\selectfont $C_2$}
		\put(83.66,7.97){\fontsize{10.69}{12.83}\selectfont $C_3$}
		\put(88.72,57.91){\fontsize{10.69}{12.83}\selectfont $C_4$}
		\put(77.25,109.87){\fontsize{10.69}{12.83}\selectfont $C_5$}
		\put(129.06,25.74){\fontsize{10.69}{12.83}\selectfont $e_1$}
		\put(143.14,110.20){\fontsize{10.69}{12.83}\selectfont $e_2$}
		\put(179.09,31.83){\fontsize{10.69}{12.83}\selectfont \textcolor[rgb]{1, 0, 0}{$C$}}
		\put(58.16,57.86){\fontsize{10.69}{12.83}\selectfont $A_5$}
		\put(287.60,57.86){\fontsize{10.69}{12.83}\selectfont $A_3$}
		\put(144.69,67.04){\fontsize{10.69}{12.83}\selectfont $e_4$}
		\put(182.26,90.48){\fontsize{10.69}{12.83}\selectfont $e_3$}
		\put(293.31,31.83){\fontsize{10.69}{12.83}\selectfont \textcolor[rgb]{1, 0, 0}{$C$}}
		\put(126.68,86.53){\fontsize{10.69}{12.83}\selectfont \textcolor[rgb]{0, 1, 0}{$\alpha$}}
		\end{picture}%
		\fi
		\caption{\label{fig:A5}%
			The link $A_5$}
	\end{figure}
	
	Take the gray test disk and denote the test complex by $U$. By symmetry and Sub-criterion \ref{prop:boundaryconnectedsum}, we only need to prove ears $e_1,e_2$ are both incompressible for $U$.

	First take test subcomplex $U'_1=U-e_2-e_4$. Then up to isotopy, there is only one arc connecting $\partial e_1$ on $D_1^+$. The union of this arc and $e_1$ is isotopic to the red circle $C$. Notice that there is a link $A_{n-2}$ as a subspace of $U'_1$, and it is easily proven that $C$ is homotopically nontrivial in the complement of $A_{n-2}$ (c.f \cite{D.Rolfsen} p.69). Thus $e_1$ is incompressible for $U'_1$.

	For $e_2$, take test subcomplex $U'_2=U-e_4$, and let $N$ be a regular neighborhood of the arc $\alpha$ on $D_1^+- \partial e_1$ connecting $\partial e_2$. Then $e_1$ is still incompressible and $\{e_2\}$ is compressible with incredible circle $\partial N$. Sub-criterion \ref{prop:incrediblecircle} asserts $\partial N$ is the unique candidate for incredible circle of $\{ e_2 \}$ for $U$. However, $e_2 \cap \alpha$ is isotopic to $C$ in the complement of $A_n$. Thus $e_2$ is incompressible.
\end{example}

Sub-criterion \ref{cor:incompressibility} and \ref{prop:incrediblecircle} apply to Fig. 4 in \cite{MM}, the examples in Exercise 4 and the 6th link in Exercise 5 in \cite{D.Rolfsen} Chapter 5 section E, Fig. 2(b) in \cite{BFJRVZ}, Fig. 2 in \cite{BCS}, Fig. 3 in \cite{F}, and Fig. 6 in \cite{M} as well.

\subsection{Proofs and some basic facts}\label{subsect:proof}\

\begin{lem}\label{lem:tangleball}
	With the notation in the paragraph before Definition \ref{def:unsplit}, suppose $\partial D_0$ bounds another disk $\tilde{D}$ outside $U'_Q$, and denote the corresponding 3-ball including $|Q|$ by $\tilde{B}$. Then there is a homeomorphism $f:(B \cup \mathbf{D_1}, U'_Q )\to (\tilde{B} \cup \mathbf{D_1}, U'_Q )$ which is fixed on $U'_Q$.
\end{lem}

If $\mathbb{R}^3$ is replaced by $\mathbb{S}^3$, this lemma is trivial and the homeomorphism is an isotopy. However in $\mathbb{R}^3$, the homeomorphism is in general not an isotopy and this lemma needs to be verified with a caution. We put the proof in the Appendix.

\begin{proof}[Proof of Sub-criterion \ref{prop:boundaryconnectedsum}]
	Assume $Q$ is compressible for $U'$ with compressing disk $D_Q$. Then $\partial D_Q$ bounds a disk $D_0$ on $D_1^+$ and the sphere $D_Q\cup D_0$ bounds a 3-ball $B\subset\mathbb{R}^3$. As $Q$ is split, there exists a properly embedded disk $D$ in $B$ which divides $Q$ into $Q_1$ and $Q_2$. Moreover, by an isotopy of $D$ in $B$, we may assume $\partial D$ consists of an arc on $D_0$ and another one on $D_Q$. Then the union of $D$ and half of $D_Q$ is a compressing disk of $Q_1$ for $U'$, a contradiction.
\end{proof}

Sub-criterion \ref{cor:incompressibility} follows immediately from the following proposition.

\begin{prop}\label{prop:trivializedsubstitution}
	With the notation in the paragraph before Sub-criterion \ref{cor:incompressibility}, the quasi-tangle $Q$ is compressible for $U'$ if and only if there is an isotopy $H : U'\times[0,1]\to  \mathbb{R}^3$ between $U'$ and the trivial substitution $\overline{U'-|Q|} \cup h(|Q|)$, such that $H(\cdot,0)=\mathrm{id}_{U'}$, and $H(\cdot,1)$ coincides with identity on $(U'-( |Q|\cup \mathbf{D_1}))\cup D_1^-$ and $h$ on $|Q|$.
\end{prop}

\begin{proof}
	If such an isotopy exists, extend it to an ambient isotopy $\mathscr{H}:\mathbb{R}^3\times [0,1]\to\mathbb{R}^3$. As $U'$ is isomorphic to $(U'-|Q|) \cup h(|Q|)$, we see $(\mathscr{H}(\cdot,1))^{-1}(\partial B_T-{\rm Int}(D_T))$ is a compressing disk of $Q$ in the complement of $U'$.
	
	For the ``only if" part, we may assume $Q$ is compressible for $U'$ with compressing disk $D_Q$. Let $B_Q$ be the 3-ball containing $|Q|$, bounded by $D_Q$ and part of $D_1^+$. Then $B_Q\cup U'$ and $B_T\cup\overline{U'-|Q|}$ are isotopic. Any orientation-preserving homeomorphism of $(B_T, D_T)$ is isotopic to identity and the isotopy can be extended. Thus there is an isotopy from $U'$ to the trivial substitution. Moreover, each isotopy involved above can be made with respect to all the restrictions.
\end{proof}

\begin{proof}[Proof of Sub-criterion \ref{prop:incrediblecircle}]
	Suppose $C,C'$ are both incredible circles of $Q$ in $U'$ and $C'$ is not isotopic to $C$ on $D_1^+$. Minimize the cardinality of $C \cap C'$ with an isotopy of $C'$ on $D_1^+$. As $C$ and $C'$ are not isotopic, now they still intersect.
	Let $D,D'$ be compressing disks in the complement of $U'$ with $\partial D=C,\partial D'=C'$. The intersection $D \cap D'$ consists of finitely many circles and arcs. The circle components can be eliminated from an innermost one on $D$, by surgery of $D$ along a subdisk in $D'$. So we suppose $D\cap D'$ is a disjoint union of finitely many arcs with endpoints in $C\cap C'$.

	Now choose one of these arcs, say $\alpha$, such that $D-\alpha$ has a component not intersecting $D'$. Denote the closure of this component by $D_\alpha$. Also, $D'-\alpha$ is a disjoint union of two disks, and we denote their closure by $D'_\alpha, \tilde{D}'_\alpha $.
	Now $\partial D_\alpha \cup \partial D'_\alpha -{\rm Int} (\alpha)$ is a circle  and bounds a disk $D_0$ on $D_1^+$. Let $B$ be the 3-ball bounded by the sphere $D_0\cup D_\alpha\cup D'_\alpha$. If $L \cap D_0=\emptyset$, we can push $D_\alpha$ across $D'_\alpha$ to eliminate $\alpha$, and the cardinality of $C \cap C'$ is reduced by 2, which contradicts our assumption. Therefore, $L\cap D_0\neq\emptyset$, and there is a quasi-tangle $Q_\alpha$ with $|Q_\alpha |\subset B\cap U'$ which is compressible for $U'$. Without loss of generality, we can assume $Q_\alpha\cap Q=\emptyset$. In fact, if $Q_\alpha\cap Q \neq\emptyset$, then $Q_\alpha=Q$ for $Q$ is unsplit, and we can replace $D'_\alpha$ in our analysis by the other disk $\tilde{D}'_\alpha$. So $Q_\alpha$ is compressible for $U'-|Q|$, which is a contradiction.
\end{proof}

We conclude this subsection with some basic facts on Brunnian links, which will be used in the next section.

\begin{lem}\label{lem:2pts}
	Let $L$ be a link and $D$ be a disk with $D \cap L =\{ p_1 , p_2 \}$. If there is another embedded  disk $D'$ with $\partial D = \partial D'$ and $D' \cap L =\emptyset$, then there exists an embedded sphere $S^2$ such that $S^2 \cap L =\{ p_1 , p_2 \}$.
\end{lem}

\begin{proof}
	Take a small regular neighbourhood $N$ of $L$. Disturb $D'$ such that $D\cap D' \cap N = \emptyset$. Then $D \cap D'$ is a disjoint union of finitely many circles. Let $C$ be an intersection circle innermost on $D'$, i.e., the disk $D'_C\subset D'$ bounded by $C$ satisfies ${\rm Int}(D'_C)\cap D=\emptyset$. We see that $C$ also bounds a disk $D_C$ on $D$, and if one of $p_1,p_2$ is in $D_C$, then both are in $D_C$. If $p_1,p_2\notin D_C$, we can replace $D$ by $(D-D_C)\cup D'_C$ and disturb it a little to eliminate the intersection circle $C$. In this way, finally for every circle $C\subset D\cap D'$, we have $p_1,p_2\in D_C$. Then we choose the circle $C_0\subset D\cap D'$, such that ${\rm Int}(D_{C_0})\cap D'=\emptyset$. The union $ D_{C_0}\cup D'_{C_0}$ is a sphere we are seeking.
\end{proof}

\begin{lem}\label{lem:connectedsum}
	Brunnian links are prime. That is, there are no nontrivial connected sum decompositions for them.
\end{lem}

\begin{proof}
	Let $L$ be a Brunnian link, and $S^2$ is an embedded sphere in $\mathbb{R}^3$ that intersects $L$ at only two points $p_1,p_2$. Up to isotopy, there is only one way to connect $p_1,p_2$ by an arc on $S_2$. Let $\alpha$ be the arc and $B$ denote the closed 3-ball bounded by $S^2$. To use proof by contradiction, suppose neither of $L_1\triangleq \alpha\cup L\cap B$ and $L_2\triangleq \alpha\cup (L- B)$ is an unknot. Then each of $L_1,L_2$ has less components than $L$.
	Note that $p_1,p_2$ must on the same component of $L$. Denote the component by $C$. As $C$ is a trivial knot, $C_1\triangleq \alpha\cup C\cap L_1, C_2\triangleq \alpha\cup C\cap L_2$ are also trivial. Therefore, $L_1$ is isotopic to the link $(L_1-C_1)\cup C$, which is a proper sublink of $L$. Thus $L_1$ bounds disjoint disks in $B$ and similarly, $L_2$ bounds disjoint disks in $\mathbb{R}^3-{\rm Int}(B)$. Now we see that $L$ bounds disjoint disks in $\mathbb{R}^3$, which is a contradiction.
\end{proof}

\begin{prop}\label{prop:2pts}
	Let $L=\cup_{i=1}^n C_i$ be a Brunnian link. Assume there is a circle $C\subset \mathbb{R}^3-L$ that bounds a disk $D$ with $D \cap L = D \cap C_1 =\{ p_1 , p_2 \}$. Suppose $\alpha\subset D$ is an arc connecting $p_1,p_2 $, and the two components of $C_1-\{p_1,p_2\}$ combined with $\alpha$ form circles $C'_1$ and $C''_1$ respectively. If neither $(L-C_1)\cup C'_1$ nor $(L-C_1)\cup C''_1$ is isotopic to $L$, then any disk bounded by $C$ intersects $L$ at two or more points.
\end{prop}

\begin{proof}
	If the linking number $lk(C, C_1 )$ equals 2, the conclusion holds. Otherwise $lk(C, C_1 )=0$.
	Suppose $C$ bounds a disk $D'$ with $D' \cap L =\emptyset$, then by Lemma \ref{lem:2pts}, there is an embedded sphere $S^2$ such that $S^2 \cap L =\{ p_1 , p_2 \}$. By Lemma \ref{lem:connectedsum}, either $(L-C_1)\cup C'_1$ or $(L-C_1)\cup C''_1$ is isotopic to $L$.
\end{proof}

\section{Series of large experiments}\label{sect:baas}

Nils A. Baas et al. constructed great many families of Brunnian links in \cite{BCS, BS}, using basic block to synthesis large Brunnian links like surfaces (see Fig. \ref{fig:Baasblock}), which generalize Brunn's chain. To show they are Brunnian, Circle-method will demonstrate significant advantages, compared with the original proofs by HOMFLY polynomial. We deal with two typical series in their construction:

\begin{figure}[htbp]
	\centering
	\includegraphics[height=2cm]{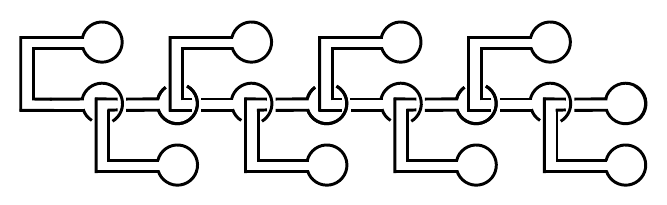}
	\caption{Baas' block}
	\label{fig:Baasblock}
\end{figure}

\textbf{Tube$(m,n)$} where $m,n \in \mathbb{N}^*$: a tube with $m$ rows and $n$ columns as shown in Fig. \ref{fig:Baas2rows} (1) and \ref{fig:Baas3rows} (1).

\textbf{Carpet$(m,n,p)$} where $m,n,p \in \mathbb{N}^*$ and $m<n$: a flat regular $p$-gon annulus in $(n-m+1)$ tiers where the innermost and outermost tiers have $mp$ and $np$ components respectively (Fig. 15 in section 3.3 in \cite{BCS}). This link has $\frac{(m+n)(n-m+1)}{2} p$ components.

We say a subset of $\mathbb{R}^3$ is \emph{split} if there is an embedded sphere in its complement that separates it into two proper subspaces. Brunnian links are not split.

\begin{example}\label{example:Baas3rows} Tube$(3,4)$, Fig. \ref{fig:Baas3rows}(1).

	\begin{figure}[htbp]
		\centering
		\ifpdf
		\setlength{\unitlength}{1bp}%
		\begin{picture}(233.32, 207.36)(0,0)
		\put(0,0){\includegraphics{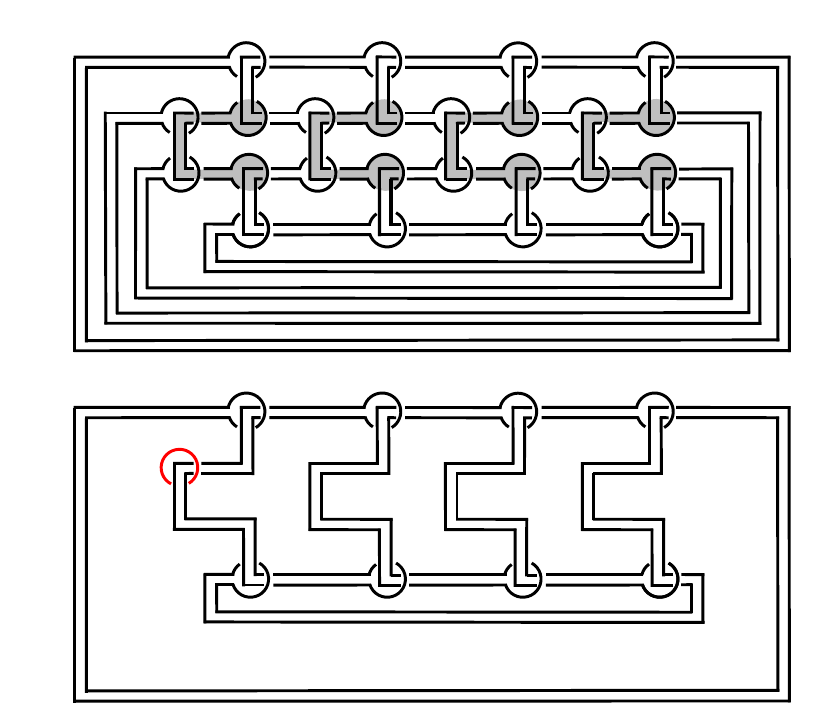}}
		\put(37.76,176.54){\fontsize{10.35}{12.42}\selectfont $a$}
		\put(37.54,193.61){\fontsize{10.35}{12.42}\selectfont $b$}
		\put(5.67,145.96){\fontsize{10.35}{12.42}\selectfont (1)}
		\put(5.67,44.85){\fontsize{10.35}{12.42}\selectfont (2)}
		\end{picture}%
		\else
		\setlength{\unitlength}{1bp}%
		\begin{picture}(233.32, 207.36)(0,0)
		\put(0,0){\includegraphics{Baas3rows}}
		\put(37.76,176.54){\fontsize{10.35}{12.42}\selectfont $a$}
		\put(37.54,193.61){\fontsize{10.35}{12.42}\selectfont $b$}
		\put(5.67,145.96){\fontsize{10.35}{12.42}\selectfont (1)}
		\put(5.67,44.85){\fontsize{10.35}{12.42}\selectfont (2)}
		\end{picture}%
		\fi
		\caption{\label{fig:Baas3rows}%
			Tube$(3,4)$}
	\end{figure}

	Take the test disks colored in gray and denote the test complex by $U$. By symmetry, there are only two kinds of ears, $a$ and $b$, whose endpoints lie on the ``back'' side and the ``front'' side respectively. If we delete all the eight ears of kind $a$ from $U$, we obtain a test subcomplex $U'$. And $U'$ has a subset, the black part of the lower figure, which is a Brunn's chain. So $U'$ is not split. Meanwhile, the trivial substitution of $b$ for $U'$ is split. Thus $b$ is incompressible for $U'$ by Sub-criterion \ref{cor:incompressibility}. The union of $a$ and any arc connecting its endpoints on the disk, is isotopic to the red circle in the complement of the Brunn's chain, which can not bound a disk according to Proposition \ref{prop:2pts}. So $a$ is also incompressible.
\end{example}

\begin{example}\label{example:Baas2rows} Tube$(2,3)$, Fig. \ref{fig:Baas2rows}(1).

	\begin{figure}[htbp]
		\centering
		\ifpdf
		\setlength{\unitlength}{1bp}%
		\begin{picture}(325.81, 171.55)(0,0)
		\put(0,0){\includegraphics{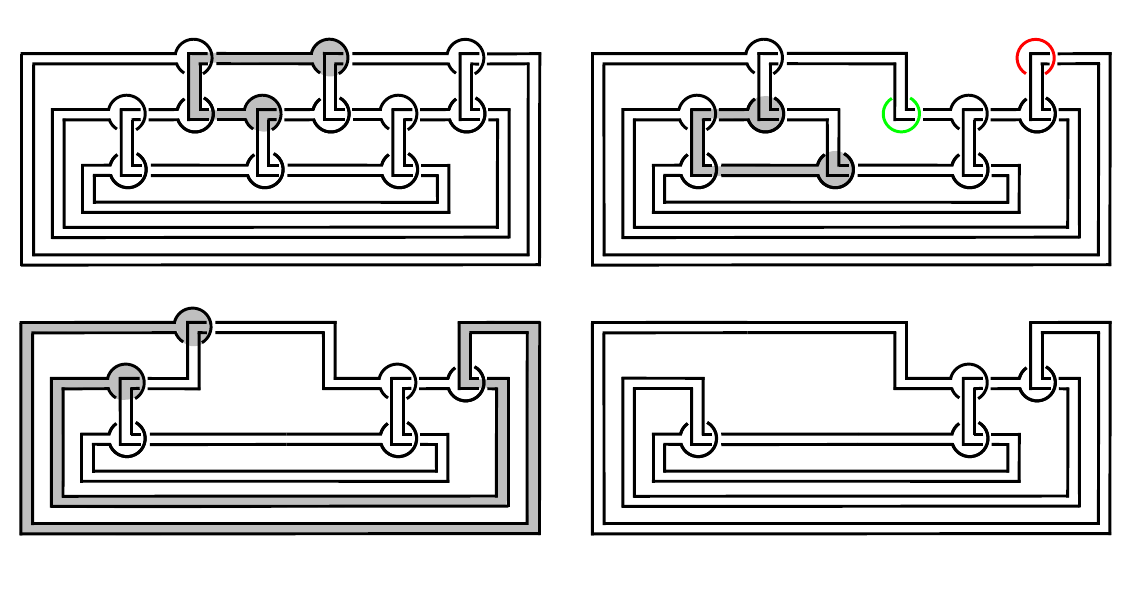}}
		\put(113.64,157.80){\fontsize{10.35}{12.42}\selectfont $a$}
		\put(106.14,143.76){\fontsize{10.35}{12.42}\selectfont $b$}
		\put(84.12,141.58){\fontsize{10.35}{12.42}\selectfont $c$}
		\put(103.96,125.61){\fontsize{10.35}{12.42}\selectfont $d$}
		\put(76.68,86.61){\fontsize{10.35}{12.42}\selectfont $(1)$}
		\put(240.96,86.61){\fontsize{10.35}{12.42}\selectfont $(2)$}
		\put(240.96,7.90){\fontsize{10.35}{12.42}\selectfont $(4)$}
		\put(76.68,7.90){\fontsize{10.35}{12.42}\selectfont $(3)$}
		\end{picture}%
		\else
		\setlength{\unitlength}{1bp}%
		\begin{picture}(325.81, 171.55)(0,0)
		\put(0,0){\includegraphics{Baas2rows}}
		\put(113.64,157.80){\fontsize{10.35}{12.42}\selectfont $a$}
		\put(106.14,143.76){\fontsize{10.35}{12.42}\selectfont $b$}
		\put(84.12,141.58){\fontsize{10.35}{12.42}\selectfont $c$}
		\put(103.96,125.61){\fontsize{10.35}{12.42}\selectfont $d$}
		\put(76.68,86.61){\fontsize{10.35}{12.42}\selectfont $(1)$}
		\put(240.96,86.61){\fontsize{10.35}{12.42}\selectfont $(2)$}
		\put(240.96,7.90){\fontsize{10.35}{12.42}\selectfont $(4)$}
		\put(76.68,7.90){\fontsize{10.35}{12.42}\selectfont $(3)$}
		\end{picture}%
		\fi
		\caption{\label{fig:Baas2rows}%
			Tube$(2,3)$}
	\end{figure}

	To show the Brunnian property of the link (1), we first claim that the black part of (2), denoted $L_{(2)}$, is a Brunnian link. Then take the gray test disk in (1) and denote the test complex by $U$. By Sub-criterion \ref{cor:incompressibility}, we only consider the unknotted ears $a,b,c,d$. Take test subcomplex $U'=U-a-c$, which takes $L_{(2)}$ as a subset. As we have claimed, $L_{(2)}$ is Brunnian thus not split, which implies $U'$ is neither split. The incompressibility of $b,d$ follows Sub-criterion \ref{cor:incompressibility} and the fact that the trivial substitution of either $b$ or $d$ for $U'$ is split. For ear $a$, take an arbitrary arc connecting $\partial a$ on the disk. The union of $a$ and the arc is isotopic to the red circle in the complement of $L_{(2)}$, which can not bound a disk in the complement of $L_{(2)}$ according to Proposition \ref{prop:2pts}. Therefore $a$ is incompressible for $U'$. Similarly, $c$ is incompressible (see the green circle). Now we only need to prove our claim that $L_{(2)}$ is Brunnian. Similarly, take the gray test disk in (2) and it suffices to show the link (3) is Brunnian. And we finally reduce to that the link (4) is Brunnian, which is a Brunn's chain with 2 components.
\end{example}

Similarly to Example \ref{example:Baas2rows} and \ref{example:Baas3rows}, one can verify the Brunnian property for Tube$(m,n)$, Brunnian annulus and Brunnian torus of any rows and columns (Fig. 12, 13 and 14 in section 3.3 in \cite{BCS}).

\begin{example}\label{example:Baas3sheaves} Carpet$(1,3,4)$, Fig. \ref{fig:Baas3sheaves}(1). 	

	\begin{figure}[htbp]
		\centering
		\ifpdf
		\setlength{\unitlength}{1bp}%
		\begin{picture}(299.53, 459.20)(0,0)
		\put(0,0){\includegraphics{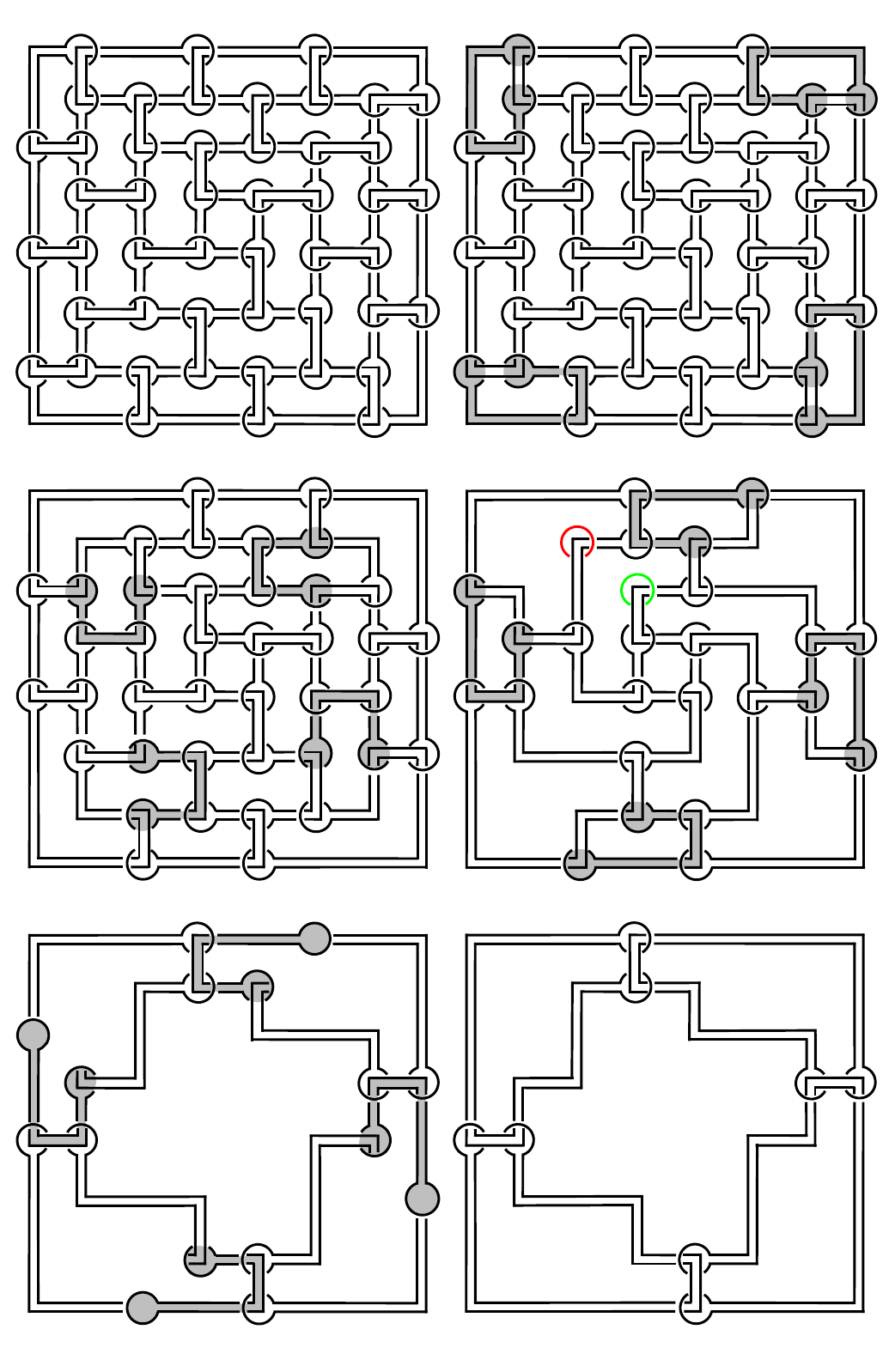}}
		\put(165.94,432.59){\fontsize{10.35}{12.42}\selectfont $a$}
		\put(177.38,432.59){\fontsize{10.35}{12.42}\selectfont $c$}
		\put(182.28,416.87){\fontsize{10.35}{12.42}\selectfont $b$}
		\put(191.26,445.45){\fontsize{10.35}{12.42}\selectfont $d$}
		\put(217.92,307.04){\fontsize{10.35}{12.42}\selectfont $(2)$}
		\put(33.38,297.02){\fontsize{10.35}{12.42}\selectfont $e$}
		\put(33.38,282.12){\fontsize{10.35}{12.42}\selectfont $f$}
		\put(56.24,282.12){\fontsize{10.35}{12.42}\selectfont $g$}
		\put(56.65,252.31){\fontsize{10.35}{12.42}\selectfont $h$}
		\put(174.61,284.18){\fontsize{10.35}{12.42}\selectfont $i$}
		\put(187.02,259.70){\fontsize{10.35}{12.42}\selectfont $j$}
		\put(28.29,135.19){\fontsize{10.35}{12.42}\selectfont $i$}
		\put(40.71,110.71){\fontsize{10.35}{12.42}\selectfont $j$}
		\put(70.91,307.04){\fontsize{10.35}{12.42}\selectfont $(1)$}
		\put(70.91,156.89){\fontsize{10.35}{12.42}\selectfont $(3)$}
		\put(217.92,156.89){\fontsize{10.35}{12.42}\selectfont $(4)$}
		\put(70.91,7.90){\fontsize{10.35}{12.42}\selectfont $(5)$}
		\put(217.92,7.90){\fontsize{10.35}{12.42}\selectfont $(6)$}
		\end{picture}%
		\else
		\setlength{\unitlength}{1bp}%
		\begin{picture}(299.53, 459.20)(0,0)
		\put(0,0){\includegraphics{Baas3sheaves}}
		\put(165.94,432.59){\fontsize{10.35}{12.42}\selectfont $a$}
		\put(177.38,432.59){\fontsize{10.35}{12.42}\selectfont $c$}
		\put(182.28,416.87){\fontsize{10.35}{12.42}\selectfont $b$}
		\put(191.26,445.45){\fontsize{10.35}{12.42}\selectfont $d$}
		\put(217.92,307.04){\fontsize{10.35}{12.42}\selectfont $(2)$}
		\put(33.38,297.02){\fontsize{10.35}{12.42}\selectfont $e$}
		\put(33.38,282.12){\fontsize{10.35}{12.42}\selectfont $f$}
		\put(56.24,282.12){\fontsize{10.35}{12.42}\selectfont $g$}
		\put(56.65,252.31){\fontsize{10.35}{12.42}\selectfont $h$}
		\put(174.61,284.18){\fontsize{10.35}{12.42}\selectfont $i$}
		\put(187.02,259.70){\fontsize{10.35}{12.42}\selectfont $j$}
		\put(28.29,135.19){\fontsize{10.35}{12.42}\selectfont $i$}
		\put(40.71,110.71){\fontsize{10.35}{12.42}\selectfont $j$}
		\put(70.91,307.04){\fontsize{10.35}{12.42}\selectfont $(1)$}
		\put(70.91,156.89){\fontsize{10.35}{12.42}\selectfont $(3)$}
		\put(217.92,156.89){\fontsize{10.35}{12.42}\selectfont $(4)$}
		\put(70.91,7.90){\fontsize{10.35}{12.42}\selectfont $(5)$}
		\put(217.92,7.90){\fontsize{10.35}{12.42}\selectfont $(6)$}
		\end{picture}%
		\fi
		\caption{\label{fig:Baas3sheaves}%
			Carpet$(1,3,4)$}
	\end{figure}

	To prove the Brunnian property of the link (1), we claim that the link (3) is Brunnian. Then take the gray test disks as shown in (2) and denote the test complex by $U_1$. There are 4 ears on each test disk and we just consider $a,b,c,d$ on test disk $D_1$. In test subcomplex $U_1-c$, we can connect the endpoints of $a$ by a unique arc on the disk. The union of the arc and $a$ has linking number with a component (in fact there are two) of link (1), thus $a$ is incompressible for $U_1$. Now suppose there is a quasi-tangle $Q\ni a$ compressible for $U_1$ with compressing disk $D_{Q}\subset \mathbb{R}^3- U_1$ and incredible circle $C\triangleq\partial D_{Q}\subset D_1^+$. If $Q\neq\{a\}$, then $U_1-a$ is split by the sphere $D_{Q}\cup C\cup D_C$, where $D_C$ is the disk on $D_1^+$ bounded by $C$ (To be exact, we shall push $D_C$ out of $\mathbf{D_1}$ a little). But the Brunnian link (3) can be viewed as a subset of $U_1-a$, which implies $U_1-a$ is not split and leads to a contradiction. $c$ is symmetric to $a$.
	Now consider $b,d$ and take test subcomplex $U'_1\triangleq U_1-a-c$. The trivial substitution of $b$ or $d$ for $U'_1$ is split. On the other hand, there is a subset of $U'_1$, the link (3) again, which is not split. So $U'_1$ is not split and by Sub-criterion \ref{cor:incompressibility}, $b,d$ are incompressible for $U'_1$. According to Sub-criterion \ref{prop:boundaryconnectedsum}, no quasi-tangle that contains $b$ or $c$ is compressible for $U_1$. So the link (1) is Brunnian.
	
	It remains to prove the Brunnian property of the link (3). Similarly we first claim the black part of (4), denoted $L_{(4)}$, is a Brunnian link. Then take the gray test disks in (3) and denote the test complex by $U_3$. Notice that $L_{(4)}$ can be viewed as a subset of $U_3$. We only consider the four unknotted ears $e,f,g,h$. In test subcomplex $U_3-h$, there is a unique arc connecting the endpoints of $f$ on the disk. The union of the arc and $f$ is isotopic to the red circle in the complement of $L_{(4)}$. According to Proposition \ref{prop:2pts} it can not bound a disk in the complement of  $L_{(4)}$, thus $f$ is incompressible for $U_3$. Considering $U_3-f$ and the green circle similarly, we see $h$ is also incompressible. For ear $e$ and $g$, take test subcomplex $U'_3\triangleq U_3-f-h$. The trivial substitution of $e$ or $h$ for $U'_3$ is split, while $L_{(4)}\subset U'_3$ implies $U'_3$ is not split. So $e,h$ are incompressible by Sub-criterion \ref{cor:incompressibility}. Also, quasi-tangles are all incompressible for $U_3$.
	
	Finally we prove $L_{(4)}$ is Brunnian by taking the gray test disks in (4) and test subcomplex $U_5$. We only need to show the ears $i,j$ are incompressible. The trivial substitution of $i$ or $j$ for $U_5$ is split. However, the link (6), which is a subset of $U_5$ and in fact a Brunn's chain of 4 components (see Example \ref{example:Milnor}), implies that $U_5$ is not split. Again by Sub-criterion \ref{cor:incompressibility} we finish the proof.
\end{example}

The Brunnian property of general Carpet$(m,n,p)$ and Brunnian solid (Fig. 16 in section 3.3 in \cite{BCS}) can be verified in a similar manner.

We now see Circle-method has three distinct features. First, it is simpler than using link invariants in that it is worked by hand. This advantage grows with the size of links. Second, it is highly efficient. For instance, during the proofs in Example \ref{example:Baas2rows} and \ref{example:Baas3sheaves}, we also show the Brunnian property of some new links as by-products. Third, it is very flexible. If we make some ``open end'' of a component go around the adjacent ``interior bend'' from the other side, i.e., conduct 4 crossing changes, or twist some double strands, the new link is still Brunnian, as we would only need a little change in the proof while retaining the whole framework. This leads to further construction of new Brunnian links.

\section{Constructions of New Brunnian links}\label{sect:construction}

It should be pointed out that the constructions below will give only an indication of its potential for constructing Brunnian links. For instance, Fig. \ref{fig:jade-pendant} illustrates a grand Brunnian link by twining Carpet$(2,3,4)$ and ``adding" copies of the true lover's knot and links $7_2^2,8_1^4$ in Rolfsen's list \cite{D.Rolfsen}. One can also try to modify other known Brunnian links to create new ones.

\begin{figure}[htbp]
	\centering
	\includegraphics[height=11cm]{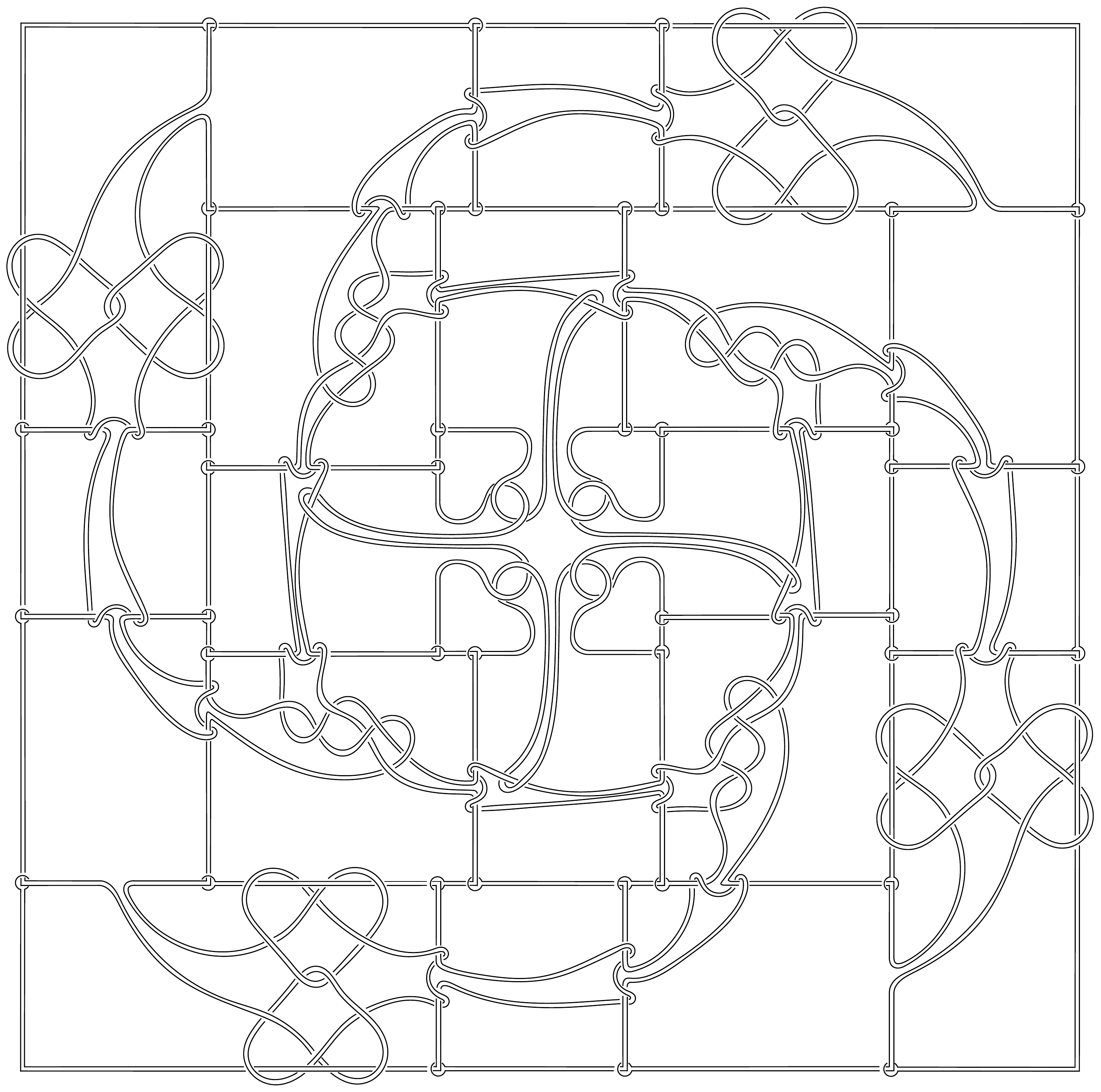}
	\caption{Jade pendant}
	\label{fig:jade-pendant}
\end{figure}

Based on each link Carpet$(m,n,p)$, we now construct several infinite families of new Brunnian links. So there are infinitely many infinite families as $m,n,p$ vary. The proof for the original link, without any change in text, shows they are Brunnian. The next three families are based on Carpet$(2,3,4)$.

\begin{enumerate}
	\item \textbf{Snake.} Choose a segment of double strands on one component, and twine it around other components as in Fig. \ref{fig:snake}. As shown in the dashed red box, we can twine it arbitrarily many times once twining happens, giving infinitely many Brannian links.
	\item \textbf{Fountains.} Twining a segment of double strands on each component around an adjacent component arbitrarily many times (as in the dashed red box) towards specific directions as in Fig. \ref{fig:fountains} gives a new sequence of Brannian links.
	\item \textbf{Cirrus.} As shown in Fig. \ref{fig:cirrus}, every component twines itself in various ways. Besides, we ``add" a tangle in the middle, which may be chosen quite arbitrarily.
\end{enumerate}

\begin{figure}[htbp]
	\centering
	\begin{minipage}[t]{0.45\textwidth}
		\centering
		\includegraphics[height=6cm]{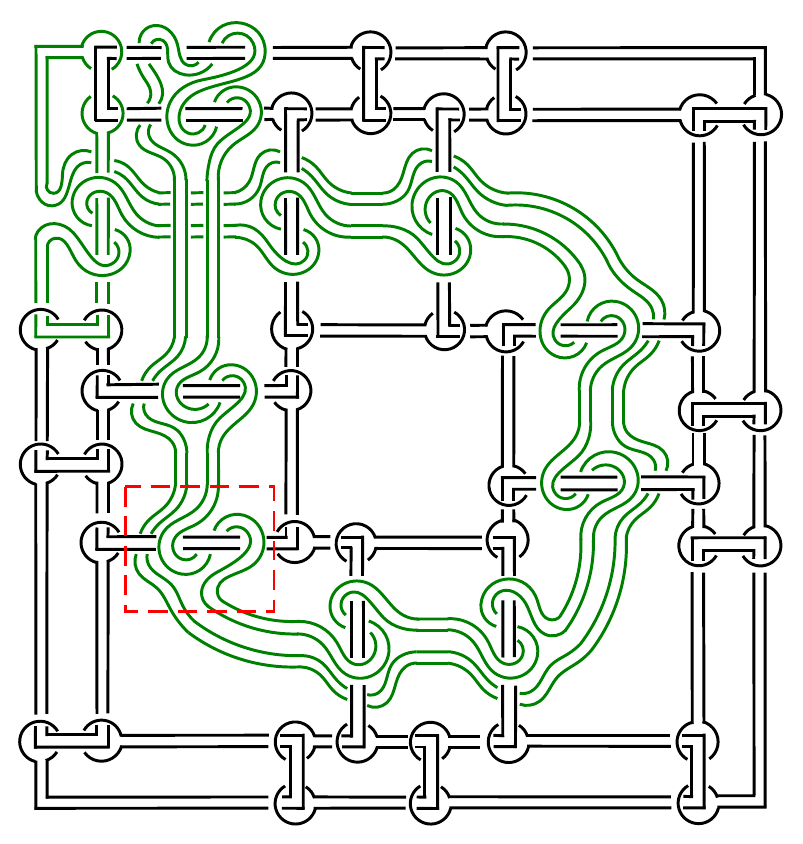}
		\caption{Snake}
		\label{fig:snake}
	\end{minipage}
	\begin{minipage}[t]{0.45\textwidth}
		\centering
		\includegraphics[height=6cm]{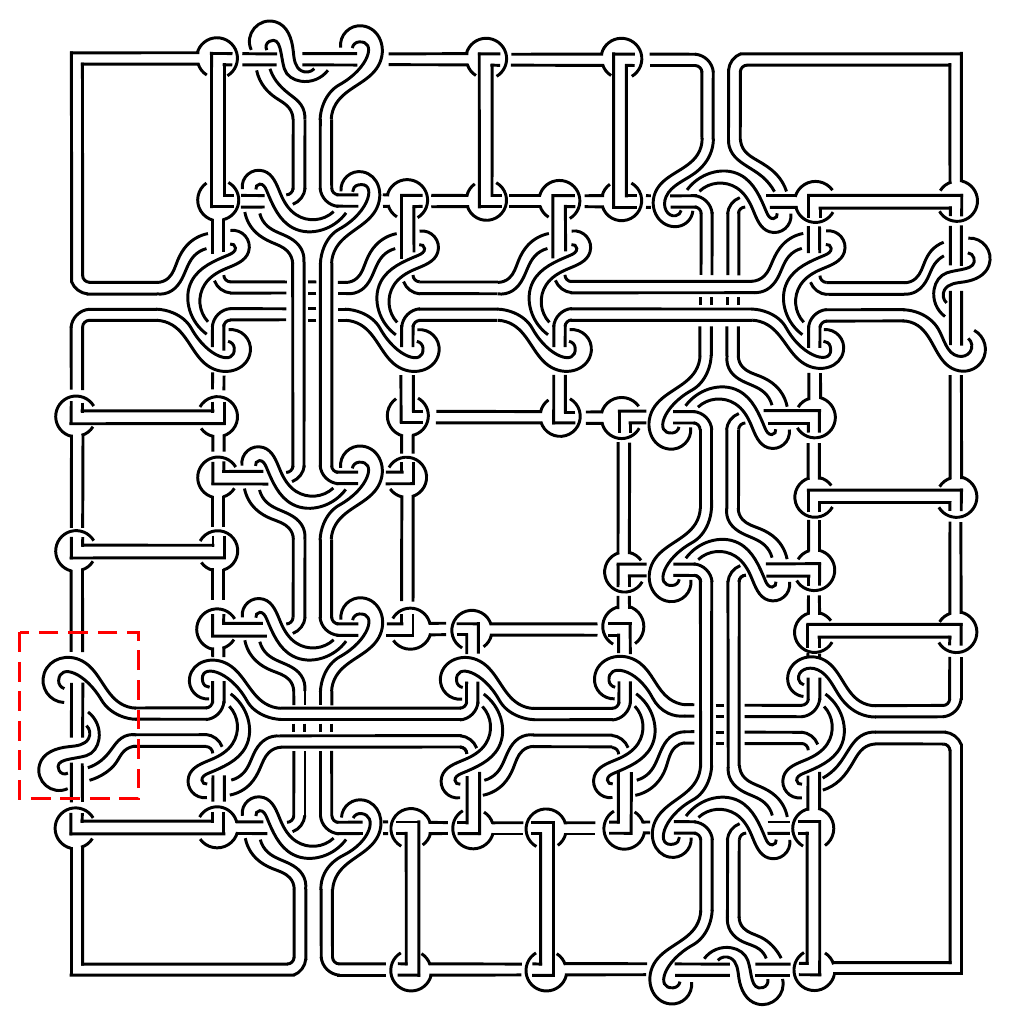}
		\caption{Fountains}
		\label{fig:fountains}
	\end{minipage}
\end{figure}

\begin{figure}[htbp]
	\centering
	\includegraphics[height=8cm]{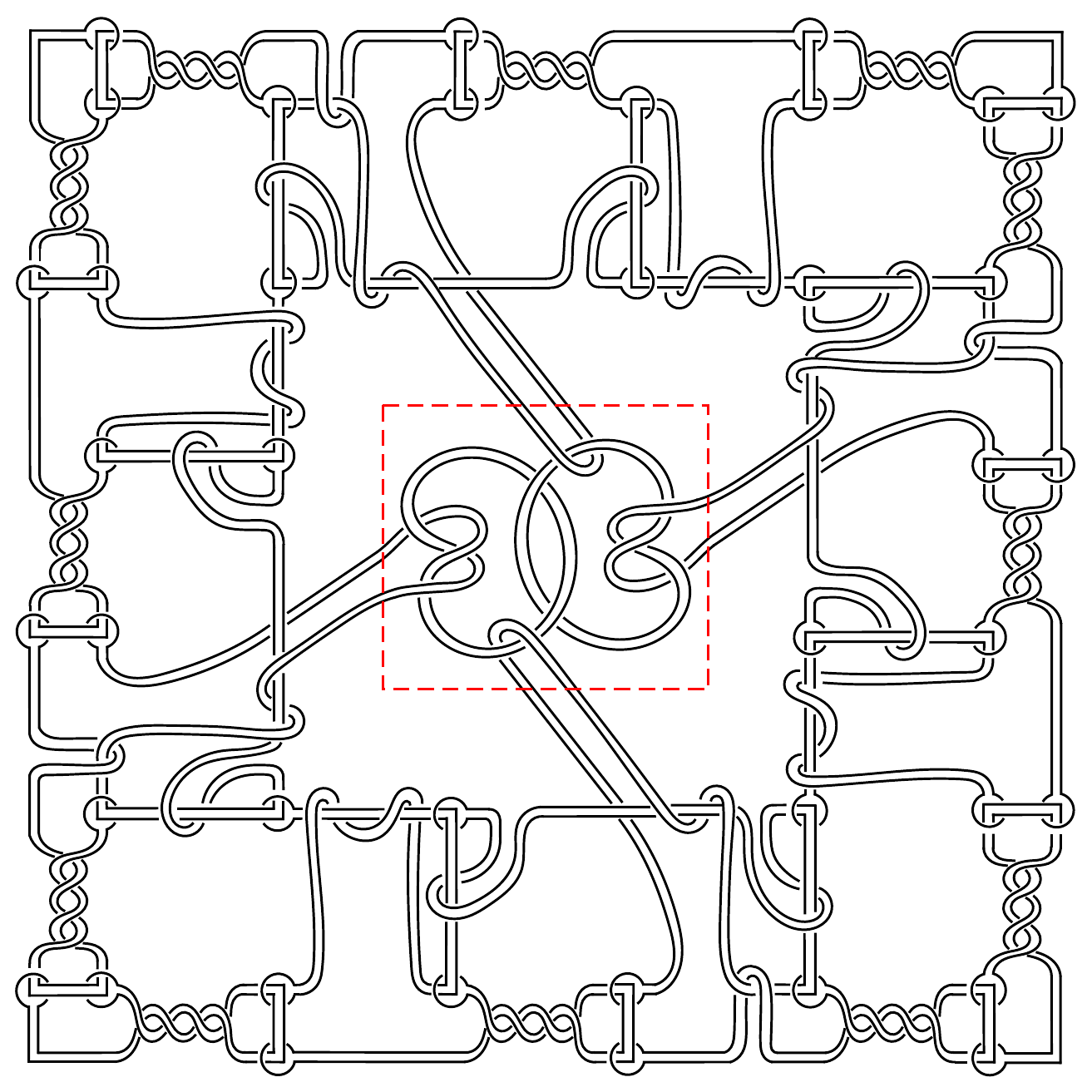}
	\caption{Cirrus}
	\label{fig:cirrus}
\end{figure}

Recall that the proofs in Example \ref{example:Baas2rows} and \ref{example:Baas3sheaves} bring new Brunnian links. We can utilize them as well to create new links. We give a series here.

\textbf{Wheel.} Fig. \ref{fig:otherclass} illustrates an infinite family of new Brunnian links by twining the components of $L_{(4)}$ in Fig. \ref{fig:Baas3sheaves}.

\begin{figure}[htbp]
	\centering
	\includegraphics[height=8cm]{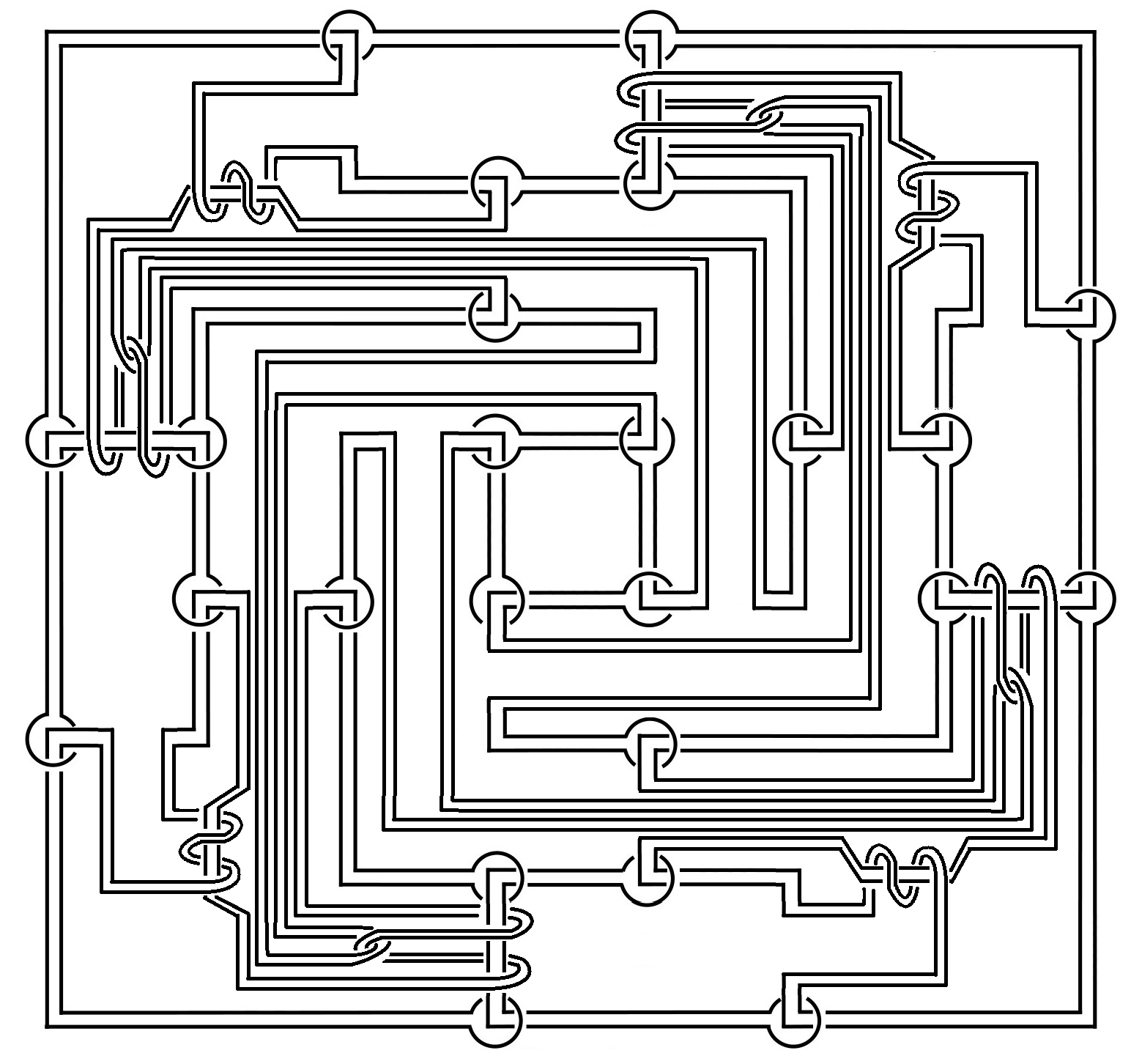}
	\caption{Wheel}
	\label{fig:otherclass}
\end{figure}

\section{Concluding Remarks}

We have established general approaches to detecting Brunnian property of links, and used them to create new Brunnian links in bulk. In fact, Arc-method and Circle-method apply for more general problems.

\begin{figure}[htbp]
	\centering
	\begin{minipage}[t]{0.3\textwidth}
		\centering
		\includegraphics[height=4cm]{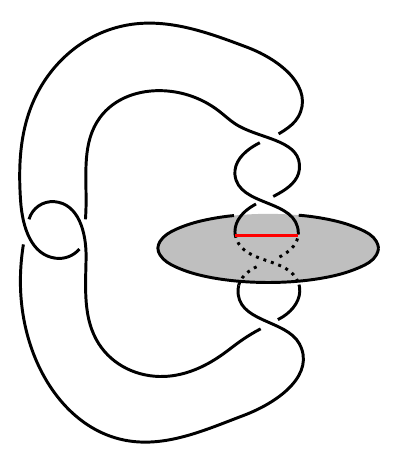}
		\caption{}
		\label{fig:notBrunnian}
	\end{minipage}
	\begin{minipage}[t]{0.3\textwidth}
		\centering
		\includegraphics[height=4cm]{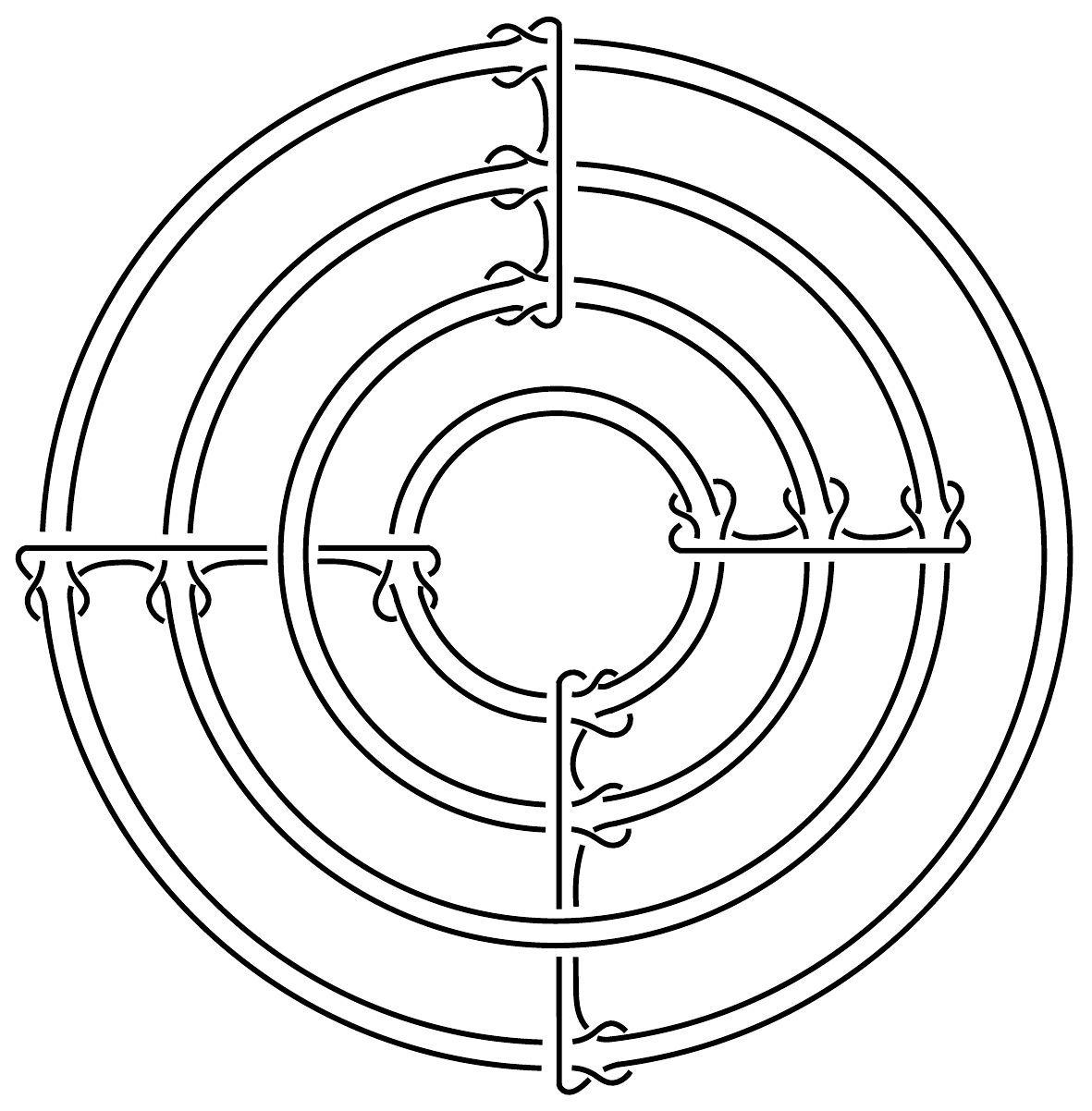}
		\caption{}
		\label{fig:brunner}
	\end{minipage}
	\begin{minipage}[t]{0.35\textwidth}
		\centering
		\includegraphics[height=4cm]{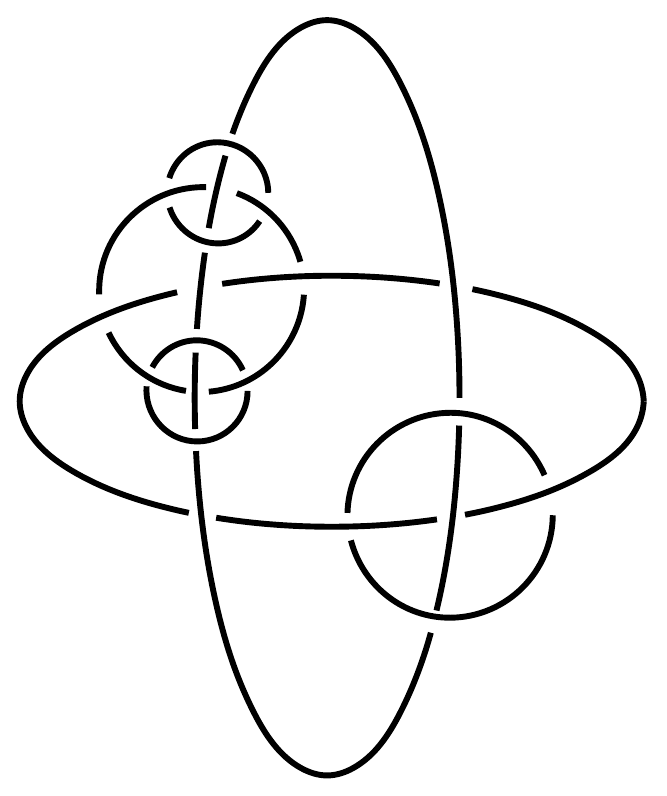}
		\caption{}
		\label{fig:fractalborromean}
	\end{minipage}
\end{figure}

1. Our method can be extended directly to show a link is nontrivial, if it has some unknot components so that we can take test disks, without considering whether each component is unknotted. For example, the test disk in Fig. \ref{fig:notBrunnian} shows fast this link is nontrivial. Especially, our methods work to detect generalized Brunnian properties (c.f. \cite{HDb,J,TKa,TK}). For instance, Arc-method shows Fig. \ref{fig:brunner} is nontrivial, which was originally proved by Alexander invariants\cite{HDb}.

2. Given a link, considering its Brunnian sublinks helps to show it is not split. Notice that Brunnian links are unsplit and if two unsplit links share a common component, then the union of them is not split either. Thus, for instance, Fig. \ref{fig:fractalborromean} (from \cite{J}) is not split.

3. In a 3-dimensional manifold $M^3$, an \emph{unlink} is defined by the boundary of disjoint union of embedded disks. So we can define Brunnian links in $M^3$ as well. It can be verified that Arc-method still holds. As a sphere in $M^3$ (even if irreducible) generally does not bound a ball on specific side, to generalize Circle-method into $M^3$, we need change the definition of ``compressible'' for quasi-tangles with care. This modification is relatively easy when $M^3$ is a submanifold of $\mathbb{S}^3$.

\vspace{12pt}

\bibliographystyle{amsalpha}

\clearpage

\section{Appendix}

\begin{figure}[htbp]
	\centering
	\ifpdf
	\setlength{\unitlength}{1bp}%
	\begin{picture}(133.66, 124.26)(0,0)
	\put(0,0){\includegraphics{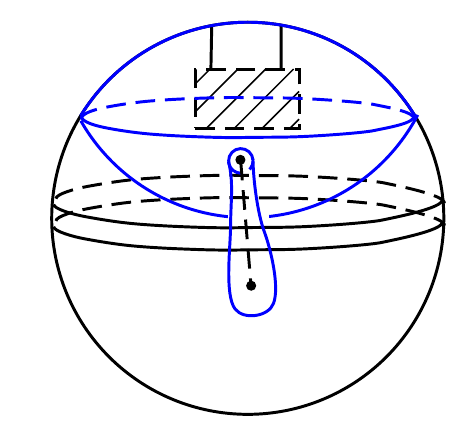}}
	\put(69.20,36.98){\fontsize{6.98}{8.38}\selectfont $\infty$}
	\put(5.67,72.70){\fontsize{6.98}{8.38}\selectfont $D_1^+$}
	\put(5.67,47.35){\fontsize{6.98}{8.38}\selectfont $D_1^-$}
	\put(35.19,88.74){\fontsize{6.98}{8.38}\selectfont $C$}
	\put(68.65,108.00){\fontsize{6.98}{8.38}\selectfont $Q$}
	\put(103.21,78.80){\fontsize{6.98}{8.38}\selectfont $D$}
	\put(108.35,107.35){\fontsize{6.98}{8.38}\selectfont $D_0$}
	\put(78.81,31.55){\fontsize{6.98}{8.38}\selectfont \textcolor[rgb]{0, 0, 1}{$B^*$}}
	\end{picture}%
	\else
	\setlength{\unitlength}{1bp}%
	\begin{picture}(133.66, 124.26)(0,0)
	\put(0,0){\includegraphics{tangleball}}
	\put(69.20,36.98){\fontsize{6.98}{8.38}\selectfont $\infty$}
	\put(5.67,72.70){\fontsize{6.98}{8.38}\selectfont $D_1^+$}
	\put(5.67,47.35){\fontsize{6.98}{8.38}\selectfont $D_1^-$}
	\put(35.19,88.74){\fontsize{6.98}{8.38}\selectfont $C$}
	\put(68.65,108.00){\fontsize{6.98}{8.38}\selectfont $Q$}
	\put(103.21,78.80){\fontsize{6.98}{8.38}\selectfont $D$}
	\put(108.35,107.35){\fontsize{6.98}{8.38}\selectfont $D_0$}
	\put(78.81,31.55){\fontsize{6.98}{8.38}\selectfont \textcolor[rgb]{0, 0, 1}{$B^*$}}
	\end{picture}%
	\fi
	\caption{\label{fig:tangleball}%
		$\overline{B_1}=\mathbb{S}^3-{\rm Int}(\mathbf{D}_1)$}
\end{figure}

\begin{proof}[Proof of Lemma \ref{lem:tangleball}]
	Consider it in $\mathbb{S}^3 = \mathbb{R}^3 \cup \{ \infty \}$, then $B_1\triangleq\mathbb{S}^3 -\mathbf{D_1}$ and $B_2\triangleq B_1-B$ are both 3-balls, see Fig. \ref{fig:tangleball}. Connect $\infty$ with a point $p\in D$ by an arc in $B_2$, and let $V$ be its regular neighborhood in $B_1$ such that $V\cap Q=\emptyset$ and $V\cap D$ is a regular neighborhood of $p$ on $D$. Let $B^*=B\cup V$. Then there is a homeomorphism $f_0:(B \cup \mathbf{D_1}, U'_Q )\to (B^* \cup \mathbf{D_1}, U'_Q )$ which is identity on $U'_Q$ and supported by a regular neighborhood of $V$. Moreover, the disk $\partial B^*\cap B_2$ is parallel to $\partial \mathbf{D}_1-D_0$ hence unique up to an isotopy in $B_1$ which fixes $Q$.
	For $\tilde{B}$ there is a corresponding 3-ball $\tilde{B}^*$. Extending the isotopy from $\partial B^*\cap B_2$ to $\partial \tilde{B}^*\cap B_2$, we obtain a homeomorphism $g:(B^* \cup \mathbf{D_1}, U'_Q )\to (\tilde{B}^* \cup \mathbf{D_1}, U'_Q )$ which is identity on $U'_Q$. $f$ can be chosen as the composition $f_0,g$ and the inverse of $f_0':(\tilde{B} \cup \mathbf{D_1}, U'_Q )\to (\tilde{B}^* \cup \mathbf{D_1}, U'_Q )$ which is defined similarly to $f_0$.
\end{proof}

\end{document}